\documentclass[11pt]{amsart}

\usepackage{graphicx,amsmath,amssymb,amsthm,bbm,url}
\usepackage{caption,subcaption}
\usepackage{fullpage}
\usepackage{color,arydshln,verbatim}
\usepackage[dvipsnames]{xcolor}
\usepackage[linesnumbered,ruled]{algorithm2e}

\newtheorem{lem}{Lemma}

\newtheorem{thm}{Theorem}

\def \f {\mathbf f}
\def \hf {\mathbf {\hat f}}

\def \e {\mathbbm e}
\def \i {\mathbbm i}

\def \n {\mathbf n}
\begin{document}
\pagestyle{plain}

\title{A New Class of Fully Discrete Sparse Fourier Transforms:  Faster Stable Implementations with Guarantees}\thanks{Sami Merhi:  Department of Mathematics, Michigan State University, East Lansing, MI, 48824, USA (merhisam@math.msu.edu).  \\ \indent Ruochuan Zhang:  Department of Mathematics, Michigan State University, East Lansing, MI, 48824, USA (zhangr12@msu.edu). \\ \indent Mark A. Iwen:  Department of Mathematics, and Department of Computational Mathematics, Science, and Engineering (CMSE), Michigan State University, East Lansing, MI, 48824, USA (markiwen@math.msu.edu).  \\ \indent Andrew Christlieb:  Department of Computational Mathematics, Science, and Engineering (CMSE), Michigan State University, East Lansing, MI, 48824, USA (andrewc@msu.edu).}
\author{Sami Merhi, Ruochuan Zhang, Mark A. Iwen, Andrew Christlieb}

\begin{abstract}
In this paper we consider Sparse Fourier Transform (SFT) algorithms for approximately computing the best $s$-term approximation of the Discrete Fourier Transform (DFT) $\hf \in \mathbb{C}^N$ of any given input vector $\f \in \mathbbm{C}^N$ in just $\left( s \log N\right)^{\mathcal{O}(1)}$-time using only a similarly small number of entries of $\f$.  %When $s \ll N$ the runtimes of these methods are significantly faster than the $\mathcal{O} \left( N \log^c N)$ runtimes of the fastest standard compressive sensing and fast Fourier transform methods.  
In particular, we present a deterministic SFT algorithm which is guaranteed to always recover a near best $s$-term approximation of the DFT of any given input vector $\f \in \mathbbm{C}^N$ in $\mathcal{O} \left( s^2 \log ^{\frac{11}{2}} (N) \right)$-time.  Unlike previous deterministic results of this kind, our deterministic result holds for both arbitrary vectors $\f \in \mathbbm{C}^N$ and vector lengths $N$.  In addition to these deterministic SFT results, we also develop several new publicly available randomized SFT implementations for approximately computing $\hf$ from $\f$ using the same general techniques.  The best of these new implementations is shown to outperform existing discrete sparse Fourier transform methods with respect to both runtime and noise robustness for large vector lengths $N$.

\end{abstract}

\maketitle
\thispagestyle{empty}

%%%%%%%%%%%%%%%%%%%%%%%%%%%%%%%%%%%%%%%%%%%%%%%%%%%%%%%
\section{Introduction}
\label{sec:intro}

Herein we are concerned with the rapid approximation of the discrete Fourier transform $\hf  \in \mathbb{C}^N$ of a given vector $\f \in \mathbb{C}^N$ for large values of $N$.  Though standard Fast Fourier Transform (FFT) algorithms  \cite{cooley1965algorithm,bluestein1970linear,rabiner1969chirp} can accomplish this task in $\mathcal{O} \left( N \log N \right)$-time for arbitrary $N \in \mathbbm{N}$, this runtime complexity may still be unnecessarily computationally taxing when $N$ is extremely large.  This is particularly true when the vector $\hf$ is approximately $s$-sparse (i.e., contains only $s \ll N$ nonzero entries) as in compressive sensing \cite{foucart2013mathematical} and certain wideband signal processing applications (see, e.g., \cite{laska2006random}).  Such applications have therefore motivated the development of Discrete Sparse Fourier Transform (DSFT) techniques \cite{gilbert2008tutorial,gilbert2014recent} which are capable of accurately approximating $s$-sparse DFT vectors $\hf \in \mathbb{C}^N$ in just $s \cdot \log^{\mathcal{O}(1)} N$-time.  When $s \ll N$ these methods are significantly faster than standard $\mathcal{O} \left( N \log N \right)$-time FFT methods, effectively achieving sublinear $o(N)$ runtime complexities in such cases.

Currently, the most widely used $s \cdot \log^{\mathcal{O}(1)} N$-time DSFT methods \cite{gilbert2005improved,iwen2007empirical,sFFT2} are randomized algorithms which accurately compute $\hf$ with high probability when given sampling access to $\f$.  Many existing sparse Fourier transforms which are entirely deterministic \cite{mk2010,Iw-arxiv,segal2013improved,lawlor2013adaptive,clw}, on the other hand, are perhaps best described as Unequally Spaced Sparse Fourier Transform (USSFT) methods in that they approximately compute $\hf$, with its entries $\hat{f}_\omega$ indexed by the set $B := \left(-\left\lceil \frac{N}{2}\right\rceil ,\left\lfloor \frac{N}{2}\right\rfloor \right]\cap\mathbb{Z}$, by sampling its associated trigonometric polynomial 
\[
f\left(x\right)=\sum_{\omega \in B}\hat{f}_{\omega} \e^{\i \omega x}
\]
at a collection of $m \ll N$ specially constructed unequally spaced points $x_1, \dots, x_m \in [-\pi, \pi]$.  These methods have no probability of failing to recover $s$-sparse $\hf$, but can not accurately compute the DFT $\hf$ of an arbitrary given vector $\f \in \mathbb{C}^N$ due to their need for \textit{unequally spaced} function evaluations from $f$ of the form $\left\{ f(x_k) \right\}^m_{k=1}$.\footnote{Note that methods which compute the DFT $\hf$ of a given vector $\f$ implicitly assume that $\f$ contains \textit{equally spaced} samples from the trigonometric polynomial $f$ above.}  

This state of affairs has left a gap in the theory of DSFT methods.  Existing deterministic sparse Fourier transform algorithms currently can efficiently compute the $s$-sparse DFT $\hf$ of a given vector $\f \in \mathbb{C}^N$ only if either $(i)$ $N$ is a power of a small prime \cite{morotti2016explicit}, or else $(ii)$ $\hat{f}_\omega = 0$ for all $\omega \in B$ with $|\omega| > N/4$ \cite{mk2010,mk2012notes}.  In this paper this gap is filled by the development of a new entirely deterministic DSFT algorithm which is always guaranteed to accurately approximate any (nearly) $s$-sparse $\hf \in \mathbb{C}^N$ of any size $N$ when given access only to $\f \in \mathbb{C}^N$.  In addition, the method used to develop this new deterministic DSFT algorithm is general enough that it can be applied to \textit{any} fast and noise robust USSFT method of the type mentioned above (be it deterministic, or randomized) in order to yield a new fast and robust DSFT algorithm.  As a result, we are also able to use the fastest of the currently existing USSFT methods \cite{iwen2008deterministic,mk2010,Iw-arxiv,segal2013improved,lawlor2013adaptive,clw,bittens2017deterministic} in order to create new publicly available DSFT implementations herein which are both faster and more robust to noise than currently existing noise robust DSFT methods for large $N$.  %More generally, we emphasize that the techniques utilized below free developers of SFT methods to develop more general USSFT methods of their choice when attempting to create better DSFT algorithms in the future by providing them with a relatively simple means of translating their USSFT results into (still fast) DSFT methods upon their completion.

More generally, we emphasize that the techniques utilized below free developers of SFT methods to develop more general USSFT methods which utilize samples from the trigonometric polynomial $f$ above at any points $\left \{ x_k \right\}^m_{k=1} \subset [-\pi,\pi]$ they like when attempting to create better DSFT algorithms in the future.  Indeed, the techniques herein provide a relatively simple means of translating any future fast and robust USSFT algorithms into (still fast) DSFT algorithms.

\subsection{Theoretical Results}  

Herein we focus on rapidly producing near best $s$-term approximations of $\hf$ of the type usually considered in compressive sensing \cite{cohen2009compressed}.  Let $\hf_{s}^{\rm opt} \in \mathbb{C}^N$ denote an optimal $s$-term approximation to $\hf \in \mathbb{C}^N$.  That is,  let $\hf_{s}^{\rm opt}$ preserve $s$ of the largest magnitude entries of $\hf$ while setting the rest of its $N-s$ smallest magnitude entires to $0$.\footnote{Note that $\hf_{s}^{\rm opt}$ may not be unique as there can be ties for the $s^{\rm th}$ largest entry in magnitude of $\f$.  This trivial ambiguity turns out not to matter.}  The following DSFT theorem is proven below.\footnote{Theorem~\ref{thm:Mainthm} is a slightly simplified version of Theorem~\ref{thm:MainTHM} proven in \S\ref{sec:ErrorBound}.}

\begin{thm}
Let $N\in\mathbb{N}$, $s\in [2,N]\cap\mathbb{N}$, $1\leq r \leq \frac{N}{36}$, and $\f \in \mathbb{C}^N$.  There exists an algorithm that will always deterministically return an $s$-sparse vector $\mathbf{v} \in\mathbb{C}^{N}$ satisfying
\begin{equation}
\left\Vert \hf-\mathbf{v} \right\Vert _{2}\leq\left\Vert \hf-\hf_{s}^{\rm opt}\right\Vert _{2}+\frac{33}{\sqrt{s}}\cdot\left\Vert \hf-\hf_{s}^{\rm opt}\right\Vert _{1}+198\sqrt{s}\left\Vert \f\right\Vert_{\infty}N^{-r} \label{equ:upper_boundIntro}
\end{equation}
in just $\mathcal{O} \left( \frac{ s^2\cdot r^{\frac32} \cdot \log ^{\frac{11}{2}} (N)} {\log (s)} \right)$-time when given access to $\f$.  If returning an $s$-sparse vector $\mathbf{v}\in\mathbb{C}^{N}$ that satisfies \eqref{equ:upper_boundIntro} for each $\f$ with probability at least $(1-p) \in [2/3,1)$ is sufficient, a Monte Carlo algorithm also exists which will do so in just $ \mathcal{O} \left(  s\cdot r^{\frac32} \cdot \log ^\frac{9}{2}(N)\cdot \log \left( \frac{N}{p}\right) \right)$-time.
\label{thm:Mainthm}
\end{thm}

Note the quadratic-in-$s$ runtime dependence of the deterministic algorithm mentioned by Theorem~\ref{thm:Mainthm}.  It turns out that there is a close relationship between the sampling points $\left \{ x_k \right\}^m_{k=1}$ used by the deterministic USSFT methods \cite{Iw-arxiv} employed as part of the proof of Theorem~\ref{thm:Mainthm} and the construction of explicit (deterministic) RIP matrices (see \cite{iwen2009simple,bailey2012design} for details).  As a result, reducing the quadratic dependence on $s$ of the $s^2 \log^{\mathcal{O}(1)} N$-runtime complexity of the deterministic DSFT algorithms referred to by Theorem~\ref{thm:Mainthm} while still satisfying the error guarantee \eqref{equ:upper_boundIntro} is likely at least as difficult as constructing explicit deterministic RIP matrices with fewer than $s^2 \log^{\mathcal{O}(1)} N$ rows by subsampling rows from an $N \times N$ DFT matrix.  Unfortunately, explicitly constructing RIP matrices of this type is known to be a very difficult problem \cite{foucart2013mathematical}.  This means that constructing an entirely deterministic DSFT algorithm which is both guaranteed to always satisfy \eqref{equ:upper_boundIntro}, and which also always runs in $s \log^{\mathcal{O}(1)} N$-time, is also likely to be extremely difficult to achieve at present.\footnote{Of course deterministic algorithms with error guarantees of the type of \eqref{equ:upper_boundIntro} do exist for more restricted classes of periodic functions $f$.  See, e.g., \cite{bittens2016sparse,bittens2017deterministic,plonka2016deterministic} for some examples. These include USSFT methods developed for periodic functions with structured Fourier support \cite{bittens2016sparse} which are of use for, among other things, the fast approximation of functions which exhibit sparsity with respect to other bounded orthonormal basis functions \cite{hu2015rapidly}.}

The remainder of this paper is organized as follows:  In section~\ref{sec:setup} we set up notation and establish necessary background results.  Then, in section~\ref{sec:Alg}, we describe our method for converting noise robust USSFT methods into DSFT methods.  The resulting approach is summarized in Algorithm~\ref{Alg1} therein.  Next, Theorem~\ref{thm:Mainthm} is proven in section~\ref{sec:ErrorBound} using the intermediary results of sections \ref{sec:setup} and \ref{sec:Alg}.  An empirical evaluation of several new DSFT algorithms resulting from our proposed approach is then performed in section \ref{sec:NumEval}.  The paper is finally concluded with a few additional comments in \S\ref{sec:Conc}.

%%%%%%%%%%%%%%%%%%%%%%%%%%%%%%%%%%%%%%%%%%%%%%%%%%%%%%%
\section{Notation and Setup}
\label{sec:setup}
The Fourier series representation of a $2\pi-$periodic function $f:\left[-\pi,\pi\right]\to\mathbb{C}$ will be denoted by
\[
f\left(x\right)=\sum_{\omega \in\mathbb{Z}}\widehat{f}_{\omega}\e^{\i \omega x}
\]
with its Fourier coefficients given by
\[
\widehat{f}_{\omega}=\frac{1}{2\pi}\int_{-\pi}^{\pi}f\left(x\right)\e^{-\i \omega x}~dx.
\]
We let $\widehat{f}:=\left\{ \widehat{f}_{\omega}\right\} _{\omega \in\mathbb{Z}}$ represent
the infinite sequence of all Fourier coefficients of $f$ below.  Given two $2\pi-$periodic
functions $f$ and $g$ we define the convolution of $f$ and $g$ at $x \in \mathbb{R}$ to be
\[
\left(f\ast g\right)\left(x\right)~=~\left(g\ast f\right)\left(x\right)~:=~ \frac{1}{2\pi}\int_{-\pi}^{\pi}g\left(x-y\right)f\left(y\right)dy.
\]
This definition, coupled with the definition of the Fourier transform,
yields the well-known equality
\[
\widehat{f\ast g}_{\omega}=\widehat{f}_{\omega}\widehat{g}_{\omega}\ \forall \omega\in\mathbb{Z}.
\]
We may also write $\widehat{f\ast g}=\widehat{f}\circ\widehat{g}$ where $\circ$
denotes the Hadamard product.

For any $N\in\mathbb{N}$, define the Discrete Fourier Transform (DFT)
matrix $F\in\mathbb{C}^{N\times N}$ by
\[
F_{\omega,j} :=\frac{(-1)^\omega}{N} \e^{-\frac{2\pi \i \cdot \omega \cdot j}{N}},
\]
and let $B:=\left(-\left\lceil \frac{N}{2}\right\rceil ,\left\lfloor \frac{N}{2}\right\rfloor \right]\cap\mathbb{Z}$
be a set of $N$ integer frequencies centered at $0$. Furthermore, let $\f \in \mathbb{C}^N$ denote the vector of equally spaced samples from $f$ whose entries are given by 
\[
f_j :=  f\left(-\pi+\frac{2\pi j}{N}\right)
\]
for $j = 0, \dots, N-1$.  One can now see that
if
\[
f\left(x\right)=\sum_{\omega \in B}\widehat{f}_{\omega} \e^{\i \omega x},
\]
then 
\begin{equation}
F \f =: \hf
\label{equ:Disc_Fourier_Def}
\end{equation}
where $\hf \in\mathbb{C}^{N}$ denotes the subset of $\widehat{f}$ with indices
in $B$, and in vector form.  More generally, bolded lower case letters will always represent vectors in $\mathbb{C}^{N}$ below.

As mentioned above, $\widehat{f}:=\left\{ \widehat{f}_{\omega}\right\} _{\omega\in\mathbb{Z}}$ is the infinite sequence of all Fourier coefficients of $f$. For any subset $S \subseteq\mathbb{Z}$ we let $\widehat{f}\vert_{S}\in\mathbb{C}^{\mathbb{Z}}$ be the sequence $\widehat{f}$ restricted to the subset $S$, so that $\widehat{f}\vert_{S}$ has terms $\left( \widehat{f}\vert_{S} \right)_\omega = \widehat{f}_{\omega}$ for all $\omega \in S$, and $\left( \widehat{f}\vert_{S} \right)_\omega = 0$ for all $\omega \in S^{c}:=\mathbb{Z}\setminus S$.  Note that $\hf$ above is exactly $\widehat{f}\vert_{B}$ excluding its zero terms for all $\omega \notin B$.  Thus, given any subset $S\subseteq B$, we let $\hf \vert_{S}\in\mathbb{C}^{N}$ be 
the vector $\hf$ restricted to the set $S$ in an analogous fashion.  That is, for $S \subseteq B$ we will have $\left( \hf \vert_{S} \right)_\omega = \hf_\omega$ for all $\omega \in S$, and $\left( \hf \vert_{S} \right)_\omega = 0$ for all $\omega \in B\setminus S$.

Given the sequence $\widehat{f}\in\mathbb{C}^{\mathbb{Z}}$ and $s\leq N$,
we denote by $R_{s}^{\rm opt}\left(\widehat{f}\right)$ a subset of $B$
containing $s$ of the most energetic frequencies of $f$; that
is
\[
R_{s}^{\rm opt}\left(\widehat{f}\right) :=\left\{ \omega_{1},\dots,\omega_{s}\right\} \subseteq B \subset \mathbb{Z}  
\]
where the frequencies $\omega_j \in B$ are ordered such that $$\left|\widehat{f}_{\omega_{1}}\right|\geq\left|\widehat{f}_{\omega_{2}}\right|\geq\cdots\geq\left|\widehat{f}_{\omega_{s}}\right|\geq\cdots \geq \left|\widehat{f}_{\omega_{N}}\right|.$$
Here, if desired, one may break ties by also requiring, e.g., that $\omega_j < \omega_k$ for all $j < k$ with $\left|\widehat{f}_{\omega_{j}}\right|=\left|\widehat{f}_{\omega_{k}}\right|$.  We will then define $f_{s}^{\rm opt}:\left[-\pi,\pi\right]\to\mathbb{C}$ based on $R_{s}^{\rm opt}\left(\widehat{f}\right)$ by
\[
f_{s}^{\rm opt}\left(x\right):=\sum_{\omega\in R_{s}^{\rm opt}\left(\widehat{f}\right)}\widehat{f}_{\omega}\e^{\i \omega x}.
\]
Any such $2 \pi$-periodic function $f_{s}^{\rm opt}$ will be referred to as an optimal $s$-term approximation to $f$.  Similarly, we also define both $\widehat{f}_{s}^{\rm opt} \in\mathbb{C}^{\mathbb{Z}}$ and $\hf_{s}^{\rm opt} \in\mathbb{C}^{N}$ to be $\widehat{f}\vert_{R_{s}^{\rm opt}\left(\widehat{f}\right)}$ and $\hf \vert_{R_{s}^{\rm opt}\left(\widehat{f}\right)}$, respectively.

\subsection{Periodized Gaussians}

In the sections that follow the $2\pi-$periodic Gaussian $g:\left[-\pi,\pi\right]\to\mathbb{R}^{+}$ defined by
\begin{equation}
g\left(x\right)=\frac{1}{c_1}\sum_{n=-\infty}^{\infty}\e^{-\frac{\left(x-2n\pi\right)^{2}}{2c_1^{2}}}
\label{equ:Def_Periodic_Gaussian}
\end{equation}
with $c_1 \in \mathbb{R}^+$ will play a special role.  The following lemmas recall several useful facts concerning both its decay, and its Fourier series coefficients. 

\begin{lem}
The $2\pi-$periodic Gaussian $g:\left[-\pi,\pi\right]\to\mathbb{R}^{+}$ has
\[
g\left(x\right) \leq \left( \frac{3}{c_1}+\frac{1}{\sqrt{2\pi}} \right)\e^{-\frac{x^{2}}{2c_{1}^{2}}}
\]
for all $x \in \left[-\pi,\pi\right]$.
\label{lem:Decay_of_Periodic_Gaussian}
\end{lem}

\begin{lem}
The $2\pi-$periodic Gaussian $g:\left[-\pi,\pi\right]\to\mathbb{R}^{+}$ has
\[
\widehat{g}_{\omega} = \frac{1}{\sqrt{2\pi}}\e^{-\frac{c_1^2 \omega^2}{2}}
\]
for all $\omega\in\mathbb{Z}$.  Thus, $\widehat{g}=\left\{ \widehat{g}_{\omega}\right\} _{\omega\in\mathbb{Z}}\in\ell^{2}$ decreases monotonically as $|\omega|$ increases, and also has $\| \widehat{g} \|_{\infty} = \frac{1}{\sqrt{2 \pi}}$.
\label{lem:Periodic_Gaussian_FC}
\end{lem}

\begin{lem}
Choose any $\tau \in \left(0, \frac{1}{\sqrt{2\pi}} \right)$, $\alpha \in \left[1, \frac{N}{\sqrt{\ln N}} \right]$, and $\beta \in \left(0 , \alpha \sqrt{\frac{\ln \left(1/\tau \sqrt{2\pi} \right)}{2}} ~\right]$. Let $c_1 = \frac{\beta \sqrt{\ln N}}{N}$ in the definition of the periodic Gaussian $g$ from \eqref{equ:Def_Periodic_Gaussian}. Then $\widehat{g}_{\omega} \in \left[\tau, \frac{1}{\sqrt{2\pi}} \right]$ for all $\omega \in \mathbb{Z}$ with $|\omega| \leq \Bigl\lceil \frac{N}{\alpha \sqrt{\ln N}}\Bigr\rceil$.
\label{lem:Periodic_Gaussian_Decay}
\end{lem}

The proofs of Lemmas~\ref{lem:Decay_of_Periodic_Gaussian},~\ref{lem:Periodic_Gaussian_FC},~and~\ref{lem:Periodic_Gaussian_Decay} are included in Appendix~\ref{sec:proof_lems_1_2} for the sake of completeness.  Intuitively, we will utilize the periodic function $g$ from \eqref{equ:Def_Periodic_Gaussian} as a bandpass filter below.  Looking at Lemma~\ref{lem:Periodic_Gaussian_Decay} in this context we can see that its parameter $\tau$ will control the effect of $\widehat{g}$ on the frequency passband defined by its parameter $\alpha$.  Deciding on the two parameters $\tau, \alpha$ then constrains $\beta$ which, in turn, fixes the periodic Gaussian $g$ by determining its constant coefficient $c_1$.  As we shall see, the parameter $\beta$ will also determine the speed and accuracy with which we can approximately sample (i.e., evaluate) the function $f \ast g$.  For this reason it will become important to properly balance these parameters against one another in subsequent sections.

\subsection{On the Robustness of the SFTs proposed in \cite{Iw-arxiv}}

The sparse Fourier transforms presented in \cite{Iw-arxiv} include both deterministic and randomized methods for approximately computing the Fourier series coefficients of a given $2 \pi$-periodic function $f$ from its evaluations at $m$-points $\left\{ x_k \right\}^m_{k=1} \subset [-\pi, \pi]$.  The following results describe how accurate these algorithms will be when they are only given approximate evaluations of $f$ at these points instead.  These results are necessary because we will want to execute the SFTs developed in \cite{Iw-arxiv} on convolutions of the form $f \ast g$ below, but will only be able to approximately compute their values at each of the required points $x_1, \dots, x_m \in [-\pi,\pi]$.

\begin{lem}
Let $s, \epsilon^{-1} \in \mathbb{N} \setminus \{ 1 \}$ with $(s/\epsilon) \geq 2$, and $\n \in \mathbb{C}^m$ be an arbitrary noise vector.  There exists a set of $m$ points $\left\{ x_k \right\}^m_{k=1} \subset [-\pi, \pi]$ such that Algorithm 3 on page 72 of \cite{Iw-arxiv}, when given access to the corrupted samples $\left\{ f(x_k) + n_k \right \}^m_{k=1}$, will identify a subset $S \subseteq B$ which is guaranteed to contain all $\omega \in B$ with
\begin{equation}
\left|\widehat{f}_{\omega} \right| > 4 \left( \frac{\epsilon \cdot \left\| \hf - \hf^{\rm opt}_{(s/\epsilon)} \right\|_1}{s} + \left\| \widehat{f} - \widehat{f}\vert_{B} \right\|_1 + \| \n \|_\infty \right).
\label{equ:GSFT_Identify}
\end{equation}
Furthermore, every $\omega \in S$ returned by Algorithm 3 will also have an associated Fourier series coefficient estimate $z_{\omega} \in \mathbb{C}$ which is guaranteed to have 
\begin{equation}
\left| \widehat{f}_{\omega} - z_{\omega} \right| \leq \sqrt{2} \left( \frac{\epsilon \cdot \left\| \hf - \hf^{\rm opt}_{(s/\epsilon)} \right\|_1}{s} + \left\| \widehat{f} - \widehat{f}\vert_{B} \right\|_1 + \| \n \|_\infty \right).
\label{equ:GSFT_Estimate}
\end{equation}
Both the number of required samples, $m$, and Algorithm 3's operation count are
\begin{equation}
\mathcal{O} \left( \frac{s^2 \cdot \log^4 (N)}{\log \left(\frac{s}{\epsilon} \right) \cdot \epsilon^2} \right).
\end{equation}

If succeeding with probability $(1-\delta) \in [2/3,1)$ is sufficient, and $(s/\epsilon) \geq 2$, the Monte Carlo variant of Algorithm 3 referred to by Corollary 4 on page 74 of \cite{Iw-arxiv} may be used.  This Monte Carlo variant reads only a randomly chosen subset of the noisy samples utilized by the deterministic algorithm,
$$\left\{ f(\tilde{x}_k) + \tilde{n}_k \right \}^{\tilde{m}}_{k=1} \subseteq \left\{ f(x_k) + n_k \right \}^m_{k=1},$$
yet it still outputs a subset $S \subseteq B$ which is guaranteed to simultaneously satisfy both of the following properties with probability at least $1-\delta$: 
\begin{enumerate}
\item [(i)] $S$ will contain all $\omega \in B$ satisfying \eqref{equ:GSFT_Identify}, and 
\item [(ii)] all $\omega \in S$ will have an associated coefficient estimate $z_{\omega} \in \mathbb{C}$ satisfying \eqref{equ:GSFT_Estimate}.
\end{enumerate}  
Finally, both this Monte Carlo variant's number of required samples, $\tilde{m}$, as well as its operation count will also always be
\begin{equation}
\mathcal{O} \left( \frac{s}{\epsilon} \cdot \log^3 (N) \cdot \log \left( \frac{N}{\delta} \right) \right).
\end{equation}
\label{lem:SFT_Ident_Coef_Est_in_Noise}
\end{lem}

Using the preceding lemma one can easily prove the following noise robust variant of Theorem~7 (and Corollary~4) from \S5 of \cite{Iw-arxiv}.  The proofs of both results are outlined in Appendix~\ref{sec:proof_lem_3_thm_2} for the sake of completeness.

\begin{thm}
Suppose $f: [-\pi,\pi] \rightarrow \mathbb{C}$ has $\widehat{f} \in \ell^1 \cap \ell^2$.  
Let $s, \epsilon^{-1} \in \mathbb{N} \setminus \{ 1 \}$ with $(s/\epsilon) \geq 2$, and $\n \in \mathbb{C}^m$ be an arbitrary noise vector.  Then, there exists a set of $m$ points $\left\{ x_k \right\}^m_{k=1} \subset [-\pi, \pi]$ together with a simple deterministic algorithm $\mathcal{A}:  \mathbb{C}^m \rightarrow \mathbb{C}^{4s}$ such that $\mathcal{A} \left( \left\{ f(x_k) + n_k \right \}^m_{k=1} \right)$ is always guaranteed to output (the nonzero coefficients of) a degree $\leq N/2$ trigonometric polynomial $y_s: [-\pi, \pi] \rightarrow \mathbb{C}$
satisfying
\begin{equation}
\left\| f - y_s \right\|_2 \leq \left\| \hf - \hf_{s}^{\rm opt} \right\|_2 + \frac{22\epsilon \cdot \left\| \hf - \hf^{\rm opt}_{(s/\epsilon)} \right\|_1}{\sqrt{s}} + 22 \sqrt{s} \left( \left\| \widehat{f} - \widehat{f}\vert_{B} \right\|_1 + \| \n \|_\infty \right).
\label{eqn:StableReconError}
\end{equation}
Both the number of required samples, $m$, and the algorithm's operation count are always
\begin{equation}
\mathcal{O} \left( \frac{s^2 \cdot \log^4 (N)}{\log \left(\frac{s}{\epsilon} \right) \cdot \epsilon^2} \right).
\label{eqn:DetRuntime}
\end{equation}

If succeeding with probability $(1-\delta) \in [2/3,1)$ is sufficient, and $(s/\epsilon) \geq 2$, a Monte Carlo variant of the deterministic algorithm may be used.  This Monte Carlo variant reads only a randomly chosen subset of the noisy samples utilized by the deterministic algorithm,
$$\left\{ f(\tilde{x}_k) + \tilde{n}_k \right \}^{\tilde{m}}_{k=1} \subseteq \left\{ f(x_k) + n_k \right \}^m_{k=1},$$
yet it still outputs (the nonzero coefficients of) a degree $\leq N/2$ trigonometric polynomial, $y_s: [-\pi, \pi] \rightarrow \mathbb{C}$, that satisfies \eqref{eqn:StableReconError} with probability at least $1-\delta$.  Both its number of required samples, $\tilde{m}$, as well as its operation count will always be
\begin{equation}
\mathcal{O} \left( \frac{s}{\epsilon} \cdot \log^3 (N) \cdot \log \left( \frac{N}{\delta} \right) \right).
\label{eqn:RandRuntime}
\end{equation}
\label{thm:StableRecov}
\end{thm}

We now have the necessary prerequisites in order to discuss our general strategy for constructing several new fully discrete SFTs.

%%%%%%%%%%%%%%%%%%%%%%%%%%%%%%%%%%%%%%%%%%%%%%%%%%%%%%%
\section{Description of the Proposed Approach}
\label{sec:Alg}
\begin{algorithm}
    \SetKwInOut{Input}{Input}
    \SetKwInOut{Output}{Output}
    
    \Input{Pointer to vector $\f \in \mathbb{C}^N$, sparsity $s \leq N$, nodes $\{ x_k \}^m_{k=1} \subset [-\pi, \pi]$ at which the given SFT algorithm $\mathcal{A}$ needs function evaluations, and $\alpha, \beta$ satisfying Lemma~\ref{lem:Periodic_Gaussian_Decay}}
    \Output{$\widehat{R}^s$, a sparse approximation of $\hf \in \mathbb{C}^N$}
    Initialize $\widehat{R}, \widehat{R}^s \leftarrow \emptyset$
    
    Set $c_1 = \frac{\beta \sqrt{\ln N}}{N}$ in the definition of periodic Gaussian $g$ from \eqref{equ:Def_Periodic_Gaussian}, and $c_2 = \frac{\alpha \sqrt{\ln N}}{2}$
    
    %Compute the sampling scheme of SFT
    
    \For{$j$ from $1$ to $\lceil c_2 \rceil$}{
        $q = -\big\lceil \frac{N}{2}\big\rceil + 1 + (2j-1)\left\lceil \frac{N}{\alpha \sqrt{\ln N}} \right\rceil$  \\
        Modulate $g$ to be $\tilde{g}_q(x) := \e^{-\i q x} g(x)$ \\
        \For{each point $x \in \{ x_k \}^m_{k=1}$}{
            Use $\f$ to approximately compute $(\tilde{g}_q\ast f)(x)$ as per \S\ref{subsec:FastConvApprox} \\
        }
        Run given SFT algorithm $\mathcal{A}$ using the approximate function evaluations $\{ (\tilde{g}_q\ast f)(x_k) \}^m_{k=1}$ in order to find an $s$-sparse Fourier approximation, $\widehat{R}_{temp} \subset \mathbb{Z} \times \mathbb{C}$, of $\widehat{\tilde{g}_q\ast f}$.\\
        \For{each (frequency,Fourier coefficient) pair $(\omega,c_{\omega}) \in \widehat{R}_{temp}$}{
            \If{$\omega \in \left[q - \left\lceil\frac{N}{\alpha \sqrt{\ln N}}\right\rceil, q + \left\lceil\frac{N}{\alpha \sqrt{\ln N}}\right\rceil \right)\cap B$}{
            $\widehat{R} = \widehat{R} ~\cup~ \left\{ \left(\omega, c_{\omega}\big/~\left(\widehat{\tilde{g}_q}\right)_\omega \right) \right\}$ \\
            } 
        }
        }

%%        \For{all $(\omega, C_{\omega})$ pairs in $R_{temp}$}{
%%            $\widehat{R} \leftarrow (\omega, \tilde{C}_{\omega})$, where $\tilde{C}_{\omega} = \sqrt{2\pi}e^{\frac{{c_1}^2 (q - \omega)^2}{2}}C_{\omega}$
%%        }
    Choose the $s$ frequencies $\omega$ with $(\omega, \tilde{c}_{\omega}) \in \widehat{R}$ having the largest $\left|\tilde{c}_{\omega}\right|$, and put those $(\omega, \tilde{c}_{\omega})$ in $\widehat{R}^s$\\
    
    Return $\widehat{R}^s$
    
    \caption{A Generic Method for Discretizing a Given SFT Algorithm $\mathcal{A}$}
    \label{Alg1}
\end{algorithm}

In this section we assume that we have access to an SFT algorithm $\mathcal{A}$ which requires $m$ function evaluations of a $2 \pi$-periodic function $f: [-\pi, \pi] \rightarrow \mathbb{C}$ in order to produce an $s$-sparse approximation to $\widehat{f}$.  For any non-adaptive SFT algorithm $\mathcal{A}$ the $m$ points $\left\{ x_k \right\}^m_{k=1} \subset [-\pi, \pi]$ at which $\mathcal{A}$ needs to evaluate $f$ can be determined before $\mathcal{A}$ is actually executed. As a result, the function evaluations $\left\{ f(x_k) \right\}^m_{k=1}$ required by $\mathcal{A}$ can also be evaluated before $\mathcal{A}$ is ever run.  Indeed, if the SFT algorithm $\mathcal{A}$ is both nonadaptive \textit{and} robust to noise it suffices to {\it approximate} the function evaluations $\left\{ f(x_k) \right\}^m_{k=1}$ required by $\mathcal{A}$ before it is executed.\footnote{We hasten to point out, moreover, that similar ideas can also be employed for {\it adaptive} and noise robust SFT algorithms in order to approximately evaluate $f$ in an ``on demand'' fashion as well.  We leave the details to the interested reader.}  These simple ideas form the basis for the proposed computational approach outlined in Algorithm~\ref{Alg1}.

The objective of Algorithm~\ref{Alg1} is to use a nonadaptive and noise robust SFT algorithm $\mathcal{A}$ which requires off-grid function evaluations in order to approximately compute the DFT of a given vector $\f \in \mathbb{C}^N$, $\hf = F \f$ .  Note that computing $\hf$ is equivalent to computing the Fourier series coefficients of the degree $N$ trigonometric interpolant of $\f$.  Hereafter the $2 \pi$-periodic function $f: [-\pi, \pi] \rightarrow \mathbb{C}$ under consideration will always be this degree $N$ trigonometric interpolant of $\f$.  Our objective then becomes to approximately compute $\widehat{f}$ using $\mathcal{A}$.  Unfortunately, our given input vector $\f$ only contains equally spaced function evaluations of $f$, and so does not actually contain the function evaluations $\left\{ f(x_k) \right\}^m_{k=1}$ required by $\mathcal{A}$.  As a consequence, we are forced to try to interpolate these required function evaluations $\left\{ f(x_k) \right\}^m_{k=1}$ from the available equally spaced function evaluations $\f$.

Directly interpolating the required function evaluations $\left\{ f(x_k) \right\}^m_{k=1}$ from $\f$ for an arbitrary degree $N$ trigonometric polynomial $f$ using standard techniques appears to be either too inaccurate, or else too slow to work well in our setting.\footnote{Each function evaluation $f(x_k)$ needs to be accurately computed in just $\mathcal{O}(\log^c N)$-time in order to allow us to achieve our overall desired runtime for Algorithm~\ref{Alg1}.}  As a result, Algorithm~\ref{Alg1} instead uses $\f$ to rapidly approximate samples from the convolution of the unknown trigonometric polynomial $f$ with (several modulations of) a known filter function $g$.  Thankfully, all of the evaluations $\left\{ (g\ast f)(x_k) \right\}^m_{k=1}$ can be approximated very accurately using only the data in $\f$ in just $\mathcal{O}(m \log N)$-time when $g$ is chosen carefully enough (see \S\ref{subsec:FastConvApprox} below).
The given SFT algorithm $\mathcal{A}$ is then used to approximate the Fourier coefficients of $g\ast f$ for each modulation of $g$ using these approximate evaluations.  Finally, $\hf$ is then approximated using the recovered sparse approximation for each $\widehat{g\ast f}$ combined with our a priori knowledge of $\widehat{g}$.

\subsection{Rapidly and Accurately Evaluating $f\ast g$}
\label{subsec:FastConvApprox}
In this section we will carefully consider the approximation of $\left(f\ast g\right)\left(x\right)$ by a {\it severely truncated version} of the semi-discrete convolution sum
\begin{equation}
\frac{1}{N}\sum^{N-1}_{j=0}f\left(-\pi+\frac{2\pi j}{N}\right)g\left(x + \pi - \frac{2\pi j}{N} \right)
\label{equ:UntruncatedFiniteSum}
\end{equation}
for any given value of $x \in [-\pi, \pi]$.    Our goal is to determine exactly how many terms of this finite sum we actually need in order to obtain an accurate approximation of $f\ast g$ at an arbitrary $x$-value.  More specifically, we aim to use as few terms from this sum as absolutely possible in order to ensure, e.g., an approximation error of size $\mathcal{O}(N^{-2})$.  

Without loss of generality, let us assume that $N=2M+1$ is odd -- this
allows us to express $B$, the set of $N$ Fourier modes about zero,
as
\[
B:=\left(-\left\lceil \frac{N}{2}\right\rceil ,\left\lfloor \frac{N}{2}\right\rfloor \right]\cap\mathbb{Z}=\left[-M,M\right]\cap\mathbb{Z}.
\]
In the lemmas and theorems below the function $f:\left[-\pi,\pi\right]\to\mathbb{C}$ will always denote a degree-$N$ trigonometric polynomial of the form
\[
f\left(x\right)=\sum_{\omega\in B}\widehat{f}_{\omega}\e^{\i\omega x}.
\]
Furthermore, $g$ will always denote the periodic Gaussian as defined above in \eqref{equ:Def_Periodic_Gaussian}. Finally, we will also
make use of the Dirichlet kernel $D_{M}:\mathbb{R}\to\mathbb{C}$,
defined by
\[
D_{M}\left(y\right)=\frac{1}{2\pi}\sum_{n=-M}^{M}\e^{\i ny}=\frac{1}{2\pi}\sum_{n\in B}\e^{\i ny}.
\]

The relationship between trigonometric polynomials such as $f$ and the Dirichlet kernel $D_{M}$ is the subject of
the following lemma.

\begin{lem}
Let $h: [-\pi, \pi] \rightarrow \mathbb{C}$ have $\widehat{h}_{\omega} = 0$ for all $\omega \notin B$, and define the set of points $\left\{ y_{j}\right\} _{j=0}^{2M}=\left\{ -\pi+\frac{2\pi j}{N} \right\} _{j=0}^{2M}$.  Then,
\[
2\pi\left(h\ast D_{M}\right)\left(x\right)~=~h\left(x\right)~=~\frac{2\pi}{N}\sum_{j=0}^{2M}h\left(y_{j}\right)D_{M}\left(x - y_{j} \right)
\]
holds for all $x\in\left[-\pi,\pi\right]$.
\label{lem:f_using_D_M}
\end{lem}
\begin{proof}
By the definition of $D_{M}$, we trivially have $2\pi \left( \widehat{D_{M}} \right)_{\omega}=\chi_{B}\left(\omega\right)$ $\forall \omega \in \mathbb{Z}$.  Thus,
\[
\widehat{h}=2\pi \cdot \widehat{h}\circ\widehat{D_{M}}=2\pi \cdot \widehat{h\ast D_{M}}
\]
where, as before, $\circ$ denotes the Hadamard product, and $\ast$
denotes convolution. This yields $h\left(x\right)=2\pi\left(h\ast D_{M}\right)\left(x\right)$ and so establishes the first equality above.
To establish the second equality above, recall from \eqref{equ:Disc_Fourier_Def} that for any $\omega\in B$ we will have
\[
\widehat{h}_{\omega}=\frac{(-1)^\omega}{N}\sum_{j=0}^{2M}h\left(-\pi+\frac{2\pi j}{N}\right)\e^{\frac{-2\pi\i j\omega}{N}}=\frac{1}{N}\sum_{j=0}^{2M}h\left(y_{j}\right)\e^{-\i\omega y_{j}},
\]
since $h$ is a trigonometric polynomial.  Thus, given $x\in\left[-\pi,\pi\right]$ one has
\begin{align*}
h\left(x\right) & =\sum_{\omega\in B}\widehat{h}_{\omega}\e^{\i\omega x} %=\frac{1}{N}\sum_{\omega\in B}\left(\sum_{j\in B}f\left(y_{j}\right)\e^{-\i\omega y_{j}}\right)\e^{\i\omega x} 
=\frac{1}{N}\sum_{j=0}^{2M}\left(h\left(y_{j}\right)\sum_{\omega\in B}\e^{\i\omega\left(x-y_{j} \right)}\right)
=\frac{2\pi}{N}\sum_{j=0}^{2M}h\left(y_{j}\right)D_{M}\left(x-y_{j} \right).
\end{align*}
We now have the desired result.
\end{proof}

We can now write a formula for $g\ast f$ which only depends on $N$ evaluations of $f$ in $[-\pi, \pi]$.

\begin{lem}
\label{lem:convolution_using_D_M}Given the set of equally spaced
points $\left\{ y_{j}\right\} ^{2M}_{j = 0}=\left\{ -\pi+\frac{2\pi j}{N} \right\} _{j=0}^{2M}$ 
one has that 
\[
\left(g\ast f\right)\left(x\right)=\frac{1}{N}\sum_{j=0}^{2M}f\left(y_{j}\right)\int_{-\pi}^{\pi}g\left(x-u-y_{j}\right)D_{M}\left(u\right)du
\]
for all $x\in\left[-\pi,\pi\right]$.
\end{lem}
\begin{proof}
By Lemma \ref{lem:f_using_D_M}, we have 
\begin{align*}
\left(g\ast f\right)\left(x\right) & =\frac{1}{2\pi}\int_{-\pi}^{\pi}g\left(x-y\right)f\left(y\right)dy =\frac{1}{N}\int_{-\pi}^{\pi}g\left(x-y\right)\sum^{2M}_{j = 0}f\left(y_{j}\right)D_{M}\left(y-y_{j}\right)dy\\
% & =\frac{1}{N}\sum^{2M}_{j = 0}f\left(y_{j}\right)\int_{-\pi-y_{j}}^{\pi-y_{j}}g\left(x-u-y_{j}\right)D_{M}\left(u\right)du\\ 
 & =\frac{1}{N}\sum^{2M}_{j = 0}f\left(y_{j}\right)\int_{-\pi}^{\pi}g\left(x-u-y_{j}\right)D_{M}\left(u\right)du.
\end{align*}
The last equality holds after a change of variables since $g$ and $D_{M}$ are both $2\pi-$periodic.
\end{proof}

The next two lemmas will help us bound the error produced by discretizing the integral weights present in the finite sum provided by Lemma~\ref{lem:convolution_using_D_M} above.  More specifically, they will ultimately allow us to approximate the sum in Lemma~\ref{lem:convolution_using_D_M} by the sum in \eqref{equ:UntruncatedFiniteSum}.

\begin{lem}
\label{lem:dirichlet_integral_as_discrete_sum}
Let $x\in\left[-\pi,\pi\right]$ and $y_{j} = -\pi+\frac{2\pi j}{N}$ for some $j = 0, \dots, 2M$.  Then,
\[
\int_{-\pi}^{\pi}g\left(x-u-y_{j}\right)D_{M}\left(u\right)du=\sum_{n\in B}\widehat{g}_{n}\e^{\i n\left(x - y_{j}\right)}.
\]
\end{lem}
\begin{proof}
Recalling that $2\pi \left( \widehat{D_{M}} \right)_{\omega}=\chi_{B}\left(\omega\right)$ for all $\omega \in \mathbb{Z}$ we have that
\begin{align*}
\int_{-\pi}^{\pi}g\left(x-u-y_{j}\right)D_{M}\left(u\right)du & = 2 \pi \left( D_{M} \ast g \right) \left( x - y_{j} \right) = \sum_{n\in \mathbb{Z}} \widehat{g}_{n} \chi_{B}\left(n\right)\e^{\i n\left(x - y_{j}\right)} = \sum_{n\in B}\widehat{g}_{n}\e^{\i n\left(x - y_{j}\right)}.
\end{align*}
\end{proof}

\begin{lem}
\label{lem:gaussian_integral_bound}Denote $I\left(a\right):=\int_{-a}^{a}\e^{-x^{2}}dx$
for $a>0$; then 
\[
\pi\left(1-\e^{-a^{2}}\right)<I^{2}\left(a\right)<\pi\left(1-\e^{-2a^{2}}\right).
\]
\end{lem}
\begin{proof}
Let $a>0$ and observe that
\begin{align*}
I^{2}\left(a\right) =\int_{-a}^{a}\int_{-a}^{a}\e^{-x^{2}-y^{2}}dx dy >\iint_{\left\{ x^{2}+y^{2}\le a^{2}\right\} }\e^{-\left(x^{2}+y^{2}\right)}dxdy 
%=\int_{0}^{2\pi}\int_{0}^{a}\e^{-r^{2}}r~drd\theta 
=\pi\left(1-\e^{-a^{2}}\right).
\end{align*}
The first equality holds by Fubini's theorem, and the inequality follows
simply by integrating a positive function over a disk of radius $a$
as opposed to a square of sidelength $2a$. A similar argument yields
the upper bound.
\end{proof}

We are now ready to bound the difference between the integral weights present in the finite sum provided by Lemma~\ref{lem:convolution_using_D_M}, and the $g\left(x - y_j \right)$-weights present in the sum \eqref{equ:UntruncatedFiniteSum}.

\begin{lem}
\label{lem:truncation_error}Choose any $\ensuremath{\tau\in\left(0,\frac{1}{\sqrt{2\pi}}\right)}$,
$\alpha\in\left[1,\frac{N}{\sqrt{\ln N}}\right]$, and $\ensuremath{\beta\in\left(0,\alpha\sqrt{\frac{\ln\left(1/\tau\sqrt{2\pi}\right)}{2}}~\right]}$.
Let $\ensuremath{c_{1}=\frac{\beta\sqrt{\ln N}}{N}}$ in the definition
of the periodic Gaussian $g$ so that
\[
g\left(x\right)=\frac{N}{\beta\sqrt{\ln N}}\sum_{n=-\infty}^{\infty}\e^{-\frac{\left(x-2n\pi\right)^{2}N^{2}}{2\beta^{2}\ln N}}.
\]
Then for all $x\in\left[-\pi,\pi\right]$ and $y_{j} = -\pi+\frac{2\pi j}{N}$,
\[
\left|g\left(x - y_{j}\right)-\int_{-\pi}^{\pi}g\left(x- u - y_{j}\right)D_{M}\left(u\right)du\right| < \frac{N^{1-\frac{\beta^{2}}{18}}}{\beta\sqrt{\ln N}}.
\]
\end{lem}
\begin{proof}
Using Lemma~\ref{lem:dirichlet_integral_as_discrete_sum} we calculate
\begin{align*}
 \left|g\left(x - y_{j}\right)-\int_{-\pi}^{\pi}g\left(x-u-y_{j}\right)D_{M}\left(u\right)du\right|& =\left|g\left(x-y_{j}\right)-\sum_{n\in B}\widehat{g}_{n}\e^{\i n\left(x - y_{j}\right)}\right| =\left|\sum_{n\in B^{c}}\widehat{g}_{n}\e^{\i n\left(x-y_{j}\right)}\right|\\
 & \leq\frac{1}{\sqrt{2\pi}}\sum_{\left|n\right|>M}\e^{-\frac{c_{1}^{2}n^{2}}{2}} {\rm \hspace{1in}(Using ~Lemma~\ref{lem:Periodic_Gaussian_FC})}\\
 & \leq\frac{2}{\sqrt{2\pi}}\int_{M}^{\infty}\e^{-\frac{c_{1}^{2}n^{2}}{2}}dn\\
 & =\sqrt{\frac{2}{\pi}}\int_{M}^{\infty}\e^{-\frac{\beta^{2}n^{2}\ln N}{2N^{2}}}dn.
\end{align*}

Upon the change of variable $v=\frac{\beta n\sqrt{\ln N}}{\sqrt{2}N}$,
we get that
\begin{align*}
 \left|g\left(x-y_{j}\right)-\int_{-\pi}^{\pi}g\left(x-u-y_{j}\right)D_{M}\left(u\right)du\right| 
& \leq\sqrt{\frac{2}{\pi}}\frac{\sqrt{2}N}{\beta\sqrt{\ln N}}\int_{\frac{\beta M\sqrt{\ln N}}{\sqrt{2}N}}^{\infty}\e^{-v^{2}}dv\\
 & =\frac{2N}{\beta\sqrt{\pi\ln N}}\frac{1}{2}\left(\int_{-\infty}^{\infty}\e^{-v^{2}}dv-\int_{-\frac{\beta M\sqrt{\ln N}}{\sqrt{2}N}}^{\frac{\beta M\sqrt{\ln N}}{\sqrt{2}N}}\e^{-v^{2}}dv\right)\\
 & <\frac{N}{\beta\sqrt{\pi\ln N}}\left(\sqrt{\pi}-\sqrt{\pi\left(1-\e^{-\frac{\beta^{2}M^{2}\ln N}{2N^{2}}}\right)}\right)\\
 & =\frac{N}{\beta\sqrt{\ln N}}\left(1-\sqrt{1-N^{-\frac{\beta^{2}M^{2}}{2N^{2}}}}\right)
\end{align*}
where the last inequality follows from Lemma \ref{lem:gaussian_integral_bound}.  Noting now that
\[
y\in\left[0,1\right]\implies1-\sqrt{1-y}\leq y,
\]
and that $\frac{N}{M}=2+\frac{1}{M}\in\left(2,3\right]$ for all $M \in \mathbb{Z}^+$, we can further see that
\begin{align*}
\frac{N}{\beta\sqrt{\ln N}}\left(1-\sqrt{1-N^{-\frac{\beta^{2}M^{2}}{2N^{2}}}}\right) \leq \frac{N}{\beta\sqrt{\ln N}}N^{-\frac{\beta^{2}M^{2}}{2N^{2}}} \leq \frac{N^{1-\frac{\beta^{2}}{18}}}{\beta\sqrt{\ln N}}
\end{align*}
also always holds.
\end{proof}

With the lemmas above we can now prove that \eqref{equ:UntruncatedFiniteSum}
can be used to approximate $\left(g\ast f\right)\left(x\right)$ for all $x \in [-\pi, \pi]$ with
controllable error.

\begin{thm}
\label{thm:convolution_error}Let $p \geq 1$.  Using the same values of the parameters
from Lemma~\ref{lem:truncation_error} above, one has
\[
\left|\left(g\ast f\right)\left(x\right)-\frac{1}{N}\sum^{2M}_{j = 0}f\left(y_{j}\right)g\left(x - y_{j}\right)\right|\leq\frac{\left\Vert \f\right\Vert _{p}}{\beta\sqrt{\ln N}}N^{1-\frac{\beta^{2}}{18} - \frac{1}{p}}
\]
for all $x\in\left[-\pi,\pi\right]$.
\end{thm}

\begin{proof}
Using Lemmas~\ref{lem:convolution_using_D_M} and~\ref{lem:truncation_error} followed by Holder's inequality, we have
\begin{align*}
 \left|\left(g\ast f\right)\left(x\right)-\frac{1}{N}\sum^{2M}_{j = 0}f\left(y_{j}\right)g\left(x-y_{j}\right)\right|
 %& =\left|\frac{1}{N}\sum_{j\in B}f\left(y_{j}\right)\int_{-\pi}^{\pi}g\left(x-u-y_{j}\right)D_{M}\left(u\right)du-\frac{1}{N}\sum_{j\in B}f\left(y_{j}\right)g\left(x-y_{j}\right)\right|\\
 &=\left|\frac{1}{N}\sum^{2M}_{j = 0}f\left(y_{j}\right)\left(g\left(x-y_{j}\right)-\int_{-\pi}^{\pi}g\left(x-u-y_{j}\right)D_{M}\left(u\right)du\right)\right|\\
 & \leq \frac{1}{N}\sum^{2M}_{j = 0}\left|f\left(y_{j}\right)\right|\frac{N^{1-\frac{\beta^{2}}{18}}}{\beta\sqrt{\ln N}} \leq \frac{N^{\frac{-\beta^{2}}{18}} }{\beta\sqrt{\ln N}}\left\Vert \f\right\Vert _{p} N^{1-\frac{1}{p}}.
\end{align*}
\end{proof}

To summarize, Theorem~\ref{thm:convolution_error} tells us that $\left(g\ast f\right)\left(x\right)$ can be approximately computed
in $\mathcal{O}\left(N\right)$-time for any $x\in\left[-\pi,\pi\right]$ using \eqref{equ:UntruncatedFiniteSum}.
This linear runtime cost may be reduced significantly, however, 
if one is willing to accept an additional trade-off between accuracy and the number
of terms needed in the sum \eqref{equ:UntruncatedFiniteSum}. This trade-off is characterized
in the next lemma. 

\begin{lem}
\label{lem:fewer_points_error}
Let $x\in\left[-\pi,\pi\right]$, $p \geq 1$, $\gamma \in \mathbb{R}^+$, and $\kappa := \lceil \gamma \ln N \rceil + 1$.  Set $j' := \arg \min_j \left| x - y_j \right|$.  Using the same values of the other parameters
from Lemma~\ref{lem:truncation_error} above, one has
\[
\left| \frac{1}{N}\sum^{2M}_{j = 0}f\left(y_{j}\right)g\left(x - y_{j}\right) - \frac{1}{N}\sum^{j' + \kappa}_{j = j' - \kappa}f\left(y_{j}\right)g\left(x - y_{j}\right) \right| \leq  2 \| \f \|_p ~N^{-\frac{2 \pi^2 \gamma^2}{\beta^2}}\]
for all $\beta \geq 4$ and $N \geq \beta^2$.
\end{lem}
\begin{proof}
Appealing to Lemma~\ref{lem:Decay_of_Periodic_Gaussian} and recalling that $c_{1}=\frac{\beta\sqrt{\ln N}}{N}$ we can see that 
\[
g\left(x\right)\leq\left(\frac{3N}{\beta\sqrt{\ln N}}+\frac{1}{\sqrt{2\pi}}\right)\e^{-\frac{x^{2}N^{2}}{2\beta^{2}\ln N}}.
\]
Using this fact we have that
\begin{align*}
g\left(x - y_{j' \pm k}\right) \leq \left(\frac{3N}{\beta\sqrt{\ln N}}+\frac{1}{\sqrt{2\pi}}\right)\e^{-\frac{\left(x - y_{j' \pm k} \right)^{2}N^{2}}{2\beta^{2}\ln N}} \leq  \left(\frac{3N}{\beta\sqrt{\ln N}}+\frac{1}{\sqrt{2\pi}}\right)\e^{-\frac{\left( 2k-1\right)^{2} \pi^2}{2\beta^{2}\ln N}} 
\end{align*}
for all $k \in \mathbb{Z}_N$.  As a result, one can now bound
\begin{align*}
\left| \frac{1}{N}\sum^{2M}_{j = 0}f\left(y_{j}\right)g\left(x - y_{j}\right) - \frac{1}{N}\sum^{j' + \kappa}_{j = j' - \kappa}f\left(y_{j}\right)g\left(x - y_{j}\right) \right| \end{align*}
above by
\begin{align}
\left(\frac{3}{\beta\sqrt{\ln N}}+\frac{1}{N \sqrt{2\pi}}\right) \sum^{N - 2 \kappa - 1}_{k = \kappa+1}  \left( \left| f\left(y_{j'-k}\right) \right| + \left| f\left(y_{j'+k}\right)   \right| \right) \e^{-\frac{\left( 2k-1\right)^{2} \pi^2}{2\beta^{2}\ln N}}, 
\label{equ:Trunc_Error__Disc_Conv_Sum}
\end{align}
where the $y_j$-indexes are considered modulo $N$ as appropriate.

Our goal is now to employ Holder's inequality on \eqref{equ:Trunc_Error__Disc_Conv_Sum}.  Toward that end, we will now bound the $q$-norm of the vector ${\bf h} := \left \{ \e^{-\frac{\left( \kappa+ \ell - \frac{1}{2}\right)^{2} 2 \pi^2}{\beta^{2}\ln N}} \right \}^{N- 2 \kappa - 1}_{\ell = 1}$.  Letting $a := q \left( \frac{4}{\beta^{2}\ln N} \right)$ we have that
\begin{align*}
\| {\bf h} \|_q^{q} &= \sum^{N- 2 \kappa - 1}_{\ell = 1} \e^{-\frac{\pi^2}{2} \left( \kappa+ \ell - \frac{1}{2}\right)^{2} a} < \sum^{\infty}_{\ell = \kappa} \e^{-\frac{\pi^2}{2} \ell^{2} a} \leq \int^\infty_{\kappa-1}  \e^{-\frac{\pi^2 x^{2}}{2} a}~dx\\
&\leq \sqrt{\frac{1}{2 \pi a}} - \frac{1}{ \pi\sqrt{2 a}} \int^{\pi (\kappa - 1) \sqrt{\frac{a}{2}}}_{-\pi (\kappa - 1) \sqrt{\frac{a}{2}}}  \e^{-u^2}~du \leq \sqrt{\frac{1}{2 \pi a}} \e^{-\frac{a \pi^2}{2}(\kappa - 1)^2} \leq \frac{\beta}{2}\sqrt{\frac{\ln N}{2 \pi q}}  N^{-\frac{2 q \pi^2 \gamma^2}{\beta^2}},
\end{align*}
where we have used  Lemma \ref{lem:gaussian_integral_bound} once again.  As a result we have that
$$\| {\bf h} \|_q \leq \left( \frac{\beta^2\ln N}{8 \pi} \right)^{\frac{1}{2q}} q^{- \frac{1}{2q}} N^{-\frac{2 \pi^2 \gamma^2}{\beta^2}} \leq \left( \frac{\beta^2\ln N}{8 \pi} \right)^{\frac{1}{2q}} N^{-\frac{2 \pi^2 \gamma^2}{\beta^2}} $$
for all $q \geq 1$.  Applying Holder's inequality on \eqref{equ:Trunc_Error__Disc_Conv_Sum} we can now see that \eqref{equ:Trunc_Error__Disc_Conv_Sum} is bounded above by
$$2 \left(\frac{3}{\beta\sqrt{\ln N}}+\frac{1}{N \sqrt{2\pi}}\right) \| \f \|_p \left( \frac{\beta^2\ln N}{8 \pi} \right)^{\frac{1}{2} - \frac{1}{2p}} N^{-\frac{2 \pi^2 \gamma^2}{\beta^2}}.$$
The result now follows.
\end{proof}

We may now finally combine the truncation and estimation errors in Theorem
\ref{thm:convolution_error} and Lemma \ref{lem:fewer_points_error}
above in order to bound the total error one incurs by approximating $\left(g\ast f\right)(x)$ via a truncated portion of \eqref{equ:UntruncatedFiniteSum} for any given $x \in [-\pi, \pi]$.

\begin{thm}
\label{thm:total_convolution_error}
Fix $x\in\left[-\pi,\pi\right]$, $p\geq 1$ (or $p = \infty$), $\frac{N}{36}\geq r \geq 1$, and $g:  [-\pi, \pi] \rightarrow \mathbb{R}^{+}$ to be the $2\pi-$periodic Gaussian \eqref{equ:Def_Periodic_Gaussian} with $c_1 := \frac{6 \sqrt{\ln(N^r)}}{N}$.  Set $j' := \arg \min_j \left| x - y_j \right|$ where $y_{j} = -\pi+\frac{2\pi j}{N}$ for all $j = 0, \dots, 2M$.  Then,
\[
\left| \left(g\ast f\right)(x) - \frac{1}{N}\sum^{j' + \left\lceil \frac{6r}{\sqrt{2} \pi} \ln N \right\rceil + 1}_{j = j' - \left\lceil \frac{6r}{\sqrt{2} \pi} \ln N \right\rceil - 1}f\left(y_{j}\right)g\left(x - y_{j}\right) \right| \leq 3 \frac{\| \f \|_p}{N^r}.
\]
As a consequence, we can see that $\left(g\ast f\right)(x)$ can always to computed to within $\mathcal{O} \left( \| \f \|_{\infty} N^{-r} \right)$-error in just $\mathcal{O}\left( r \log N \right)$-time for any given $\f \in \mathbb{C}^N$ once the $\big\{ g\left(x - y_{j}\right) \big\}^{j' + \left\lceil \frac{6r}{\sqrt{2} \pi} \ln N \right\rceil + 1}_{j = j' - \left\lceil \frac{6r}{\sqrt{2} \pi} \ln N \right\rceil - 1}$ have been precomputed.  
\end{thm}

\begin{proof}
Combining Theorem \ref{thm:convolution_error} and Lemma \ref{lem:fewer_points_error} we can see that 
\[
\left| \left(g\ast f\right)(x) - \frac{1}{N}\sum^{j' + \left\lceil \frac{6r}{\sqrt{2} \pi} \ln N \right\rceil + 1}_{j = j' - \left\lceil \frac{6r}{\sqrt{2} \pi} \ln N \right\rceil - 1}f\left(y_{j}\right)g\left(x - y_{j}\right) \right| \leq \| \f \|_p \left( \frac{1}{\beta\sqrt{\ln N}}N^{1-\frac{\beta^{2}}{18} - \frac{1}{p}} + 2 ~N^{-\frac{2 \pi^2 \gamma^2}{\beta^2}} \right)
\]
where $\beta = 6 \sqrt{r} \geq 6$, $N \geq 36 r = \beta^2$, and $\gamma = \frac{6r}{\sqrt{2} \pi} = \frac{\beta \sqrt{r}}{\sqrt{2} \pi}$.  
\end{proof}

We are now prepared to bound the error of the proposed approach when utilizing the SFTs developed in \cite{Iw-arxiv}.

%%%%%%%%%%%%%%%%%%%%%%%%%%%%%%%%%%%%%%%%%%%%%%%%%%%%%%%
\section{An Error Guarantee for Algorithm~\ref{Alg1} when Using the SFTs Proposed in \cite{Iw-arxiv}}
\label{sec:ErrorBound}
Given the $2\pi-$periodic Gaussian $g:  [-\pi, \pi] \rightarrow \mathbb{R}^{+}$ \eqref{equ:Def_Periodic_Gaussian}, consider the periodic modulation of $g$, $\tilde{g}_{q}:\left[-\pi,\pi\right]\to\mathbb{C}$, for any $q\in\mathbb{Z}$ defined by
\[
\tilde{g}_{q}\left(x\right)= \e ^{-\i qx}g\left(x\right).
\]
One can see that
\begin{align*}
\tilde{g}_{q}\left(x\right) & =\e ^{-\i qx}\sum_{\omega=-\infty}^{\infty}\widehat{g}_{\omega} \e ^{\i \omega x}=\sum_{\omega=-\infty}^{\infty}\widehat{g}_{\omega} \e ^{ \i \left(\omega-q\right)x}=\sum_{\tilde{\omega}=-\infty}^{\infty}\widehat{g}_{\tilde{\omega}+q} \e ^{ \i \tilde{\omega}x},
\end{align*}
so that the Fourier series coefficients of $\tilde{g}_{q}$ are those of $g$, shifted by $q$; that is,
\[
\left(\widehat{\tilde{g}_{q}}\right)_{\omega}=\widehat{g}_{\omega+q}.
\]

In line 9 of Algorithm~\ref{Alg1}, we provide the SFT Algorithm in \cite{Iw-arxiv} with the approximate evaluations of $\left\{ \left(\tilde{g}_{q}\ast f\right)\left(x_{k}\right)\right\} _{k=1}^{m},$ namely, $\left\{ \left(\tilde{g}_{q}\ast f\right)\left(x_{k}\right)+n_{k}\right\} _{k=1}^{m}$, where, by Theorem \ref{thm:total_convolution_error}, the perturbations $n_{k}$ are bounded, for instance, by
\[
\left|n_{k}\right|\leq3\frac{\left\Vert f\right\Vert _{\infty}}{N^{r}}\ \text{\ensuremath{\forall}\ }k=1,\dots,m.
\]

With this in mind, let us apply Lemma \ref{lem:SFT_Ident_Coef_Est_in_Noise} to the function $\tilde{g}_{q}\ast f$. We have the following lemma.

\begin{lem}
Let $s\in [2,N]\cap\mathbb{N}$, and $\mathbf{n}\in\mathbb{C}^{m}$ be the vector containing the total errors incurred by approximating $\tilde{g}_{q}\ast f$ via a truncated version of \eqref{equ:UntruncatedFiniteSum}, as per Theorem \ref{thm:total_convolution_error}. There exists a set of $m$ points $\left\{ x_{k}\right\} _{k=1}^{m}\subset\left[-\pi,\pi\right]$ such that Algorithm 3 on page 72 of \cite{Iw-arxiv}, when given access to the corrupted samples $\left\{ \left(\tilde{g}_{q}\ast f\right)\left(x_{k}\right)+n_{k}\right\} _{k=1}^{m},$ will identify a subset $S\subseteq B$ which is guaranteed to contain all $\omega\in B$ with
\[
\left|\left(\widehat{\tilde{g}_{q}\ast f}\right)_{\omega}\right|>4\left(\frac{1}{s}\cdot\left\Vert \widehat{\tilde{g}_{q}\ast f}-\left(\widehat{\tilde{g}_{q}\ast f}\right)_{s}^{\rm opt}\right\Vert _{1}+3\left\Vert \f\right\Vert _{\infty}N^{-r}\right)=:4\tilde{\delta}.
\]
Furthermore, every $\omega\in S$ returned by Algorithm 3 will also have an associated Fourier series coefficient estimate $z_{\omega}\in\mathbb{C}$ which is guaranteed to have
\[
\left|\left(\widehat{\tilde{g}_{q}\ast f}\right)_{\omega}-z_{\omega}\right|\leq\sqrt{2}\tilde{\delta}.
\]
\end{lem}

Next, we need to guarantee that the estimates of $\widehat{\tilde{g}_{q}\ast f}$ returned by Algorithm 3 of \cite{Iw-arxiv} will yield good estimates of $\widehat{f}$ itself. We have the following.

\begin{lem}
Let $s\in [2,N]\cap\mathbb{N}$. Given a $2\pi-$periodic function $f:\left[-\pi,\pi\right]\rightarrow\mathbb{C}$, the periodic Gaussian $g$, and any of its modulations $\tilde{g}_{q}\left(x\right)=\e^{-\i qx}g\left(x\right)$, one has
\[
\left\Vert \widehat{\tilde{g}_{q}\ast f}-\left(\widehat{\tilde{g}_{q}\ast f}\right)_{s}^{\rm opt}\right\Vert _{1}\leq \frac{1}{2} \left\Vert \widehat{f}-\widehat{f}_{s}^{\rm opt}\right\Vert _{1}.
\]
\end{lem}
\begin{proof}

Recall the definition of $R_{s}^{\rm opt}\left(\widehat{f}\right)$ as the subset of $B$ containing the $s$ most energetic frequencies of $\widehat{f}$, and observe that
\[
\frac{1}{2} \left\Vert \widehat{f}-\widehat{f}_{s}^{\rm opt}\right\Vert _{1}= \frac{1}{2} \sum_{\omega\in B\backslash R_{s}^{\rm opt}\left(\widehat{f}\right)}\left|\widehat{f}_{\omega}\right|\geq\sum_{\omega\in B\backslash R_{s}^{\rm opt}\left(\widehat{f}\right)}\left|\left(\widehat{\tilde{g}_{q}}\right)_{\omega}\cdot\widehat{f}_{\omega}\right|
\]
since, by Lemma \ref{lem:Periodic_Gaussian_FC}, $\widehat{g}_{\omega}<\frac{1}{2}$ for all $\omega$, and consequently, $\left(\widehat{\tilde{g}_{q}}\right)_{\omega}=\widehat{g}_{\omega+q}<\frac{1}{2}$ for all $\omega$. Moreover,
\begin{align*}
\sum_{\omega\in B\backslash R_{s}^{\rm opt}\left(\widehat{f}\right)}\left|\left(\widehat{\tilde{g}_{q}}\right)_{\omega}\cdot\widehat{f}_{\omega}\right| & \geq\sum_{\omega\in B\backslash R_{s}^{\rm opt}\left(\widehat{\tilde{g}_{q}\ast f}\right)}\left|\left(\widehat{\tilde{g}_{q}}\right)_{\omega}\cdot\widehat{f}_{\omega}\right|
  ~=~\left\Vert \widehat{\tilde{g}_{q}\ast f}-\left(\widehat{\tilde{g}_{q}\ast f}\right)_{s}^{\rm opt}\right\Vert _{1}.
\end{align*}
\end{proof}

Let us combine the guarantees above into the following lemma.

\begin{lem}
Let $s\in [2,N]\cap\mathbb{N}$, and $\mathbf{n}\in\mathbb{C}^{m}$ be the vector containing the total errors incurred by approximating $\tilde{g}_{q}\ast f$ via a truncated version of \eqref{equ:UntruncatedFiniteSum}, as per Theorem \ref{thm:total_convolution_error}. There exists a set of $m$ points $\left\{ x_{k}\right\} _{k=1}^{m}\subset\left[-\pi,\pi\right]$ such that Algorithm 3 on page 72 of \cite{Iw-arxiv}, when given access to the corrupted samples $\left\{ \left(\tilde{g}_{q}\ast f\right)\left(x_{k}\right)+n_{k}\right\} _{k=1}^{m},$ will identify a subset $S\subseteq B$ which is guaranteed to contain all $\omega\in B$ with
\[
\left|\left(\widehat{\tilde{g}_{q}\ast f}\right)_{\omega}\right|>4\left(\frac{1}{2s}\cdot\left\Vert \widehat{f}-\widehat{f}_{s}^{\rm opt}\right\Vert _{1}+3\left\Vert \f\right\Vert _{\infty}N^{-r}\right)=:4\delta.
\]
Furthermore, every $\omega\in S$ returned by Algorithm 3 will also have an associated Fourier series coefficient estimate $z_{\omega}\in\mathbb{C}$ which is guaranteed to have
\[
\left|\left(\widehat{\tilde{g}_{q}}\right)_{\omega}\cdot\widehat{f}_{\omega}-z_{\omega}\right|\leq\sqrt{2}\delta.
\]
\end{lem}

The lemma above implies that for any choice of $q$ in line 4 of Algorithm~\ref{Alg1}, we are guaranteed to find all $\omega\in\left[q-\left\lceil \frac{N}{\alpha\sqrt{\ln N}}\right\rceil ,q+\left\lceil \frac{N}{\alpha\sqrt{\ln N}}\right\rceil \right)\cap B$ with
\[
\left|\widehat{f}_{\omega}\right|>\max_{\tilde{\omega}}\frac{4\delta}{\left(\widehat{\tilde{g}_{q}}\right)_{\tilde{\omega}}}\geq\frac{4\delta}{\tau}
\]
where $\alpha$ and $\tau$ are as defined in Lemma \ref{lem:Periodic_Gaussian_Decay}. Moreover, the Fourier series coefficient estimates $z_{\omega}$ returned by Algorithm 3 will satisfy
\[
\left|\widehat{f}_{\omega}-\frac{z_{\omega}}{\left(\widehat{\tilde{g}_{q}}\right)_{\omega}}\right|\leq\max_{\tilde{\omega}}\frac{\sqrt{2}\delta}{\left(\widehat{\tilde{g}_{q}}\right)_{\tilde{\omega}}}\leq\frac{\sqrt{2}\delta}{\tau}.
\]

Following Theorem 3, which guarantees a decay of $N^{-r}$ in the total approximation error, let us set $\beta=6\sqrt{r}$ for $1\leq r\leq\frac{N}{36}$. Recall from Lemma \ref{lem:Periodic_Gaussian_Decay} the choice of $\beta$ $\in\left(0,\alpha\sqrt{\frac{\ln\left(1/\tau\sqrt{2\pi}\right)}{2}}\right]$ where $\tau$ is to be chosen from $\left(0,\frac{1}{\sqrt{2\pi}}\right)$. Thus, we must choose $\alpha\in\left[1,\frac{N}{\sqrt{\ln N}}\right]$ so that
\[
6\sqrt{r}\leq\alpha\sqrt{\frac{\ln\left(1/\tau\sqrt{2\pi}\right)}{2}}\iff\alpha\geq\frac{6\sqrt{2r}}{\ln\left(1/\tau\sqrt{2\pi}\right)}.
\]
We may remove the dependence on $\tau$ simply by setting, e.g., $\tau=\frac{1}{3}$. Then $\alpha=\mathcal{O}\left(\sqrt{r}\right)$.

We are now ready to state the recovery guarantee of Algorithm ~\ref{Alg1} and its operation count.

\begin{thm}
Let $N\in\mathbb{N}$, $s\in [2,N]\cap\mathbb{N}$, and $1\leq r \leq \frac{N}{36}$ as in Theorem \ref{thm:total_convolution_error}. If Algorithm 3 of \cite{Iw-arxiv} is used in Algorithm~\ref{Alg1} then Algorithm~\ref{Alg1} will always deterministically identify a subset $S\subseteq B$ and a sparse vector $\mathbf{v}\vert_{S}\in\mathbb{C}^{N}$ satisfying
\begin{equation}
\left\Vert \hf-\mathbf{v}\vert_{S}\right\Vert _{2}\leq\left\Vert \hf-\hf_{s}^{\rm opt}\right\Vert _{2}+\frac{33}{\sqrt{s}}\cdot\left\Vert \hf-\hf_{s}^{\rm opt}\right\Vert _{1}+198\sqrt{s}\left\Vert \f\right\Vert_{\infty}N^{-r}.\label{equ:upper_bound}
\end{equation}
Algorithm ~\ref{Alg1}'s operation count is then
$$ \mathcal{O} \left( \frac{ s^2\cdot r^{\frac32} \cdot \log ^{\frac{11}{2}} (N)} {\log (s)} \right).$$

If returning a sparse vector $\mathbf{v}\vert_{S}\in\mathbb{C}^{N}$ that satisfies \eqref{equ:upper_bound} with probability at least $(1-p) \in [2/3,1)$ is sufficient, a Monte Carlo variant of the deterministic Algorithm 3 in \cite{Iw-arxiv} may be used in line 9 of Algorithm ~\ref{Alg1}. In this case Algorithm ~\ref{Alg1}'s operation count is
$$ \mathcal{O} \left(  s\cdot r^{\frac32} \cdot \log ^\frac{9}{2}(N)\cdot \log \left( \frac{N}{p}\right) \right).$$
\label{thm:MainTHM}
\end{thm}

\begin{proof}
Redefine $\delta$ in the proof of Theorem 7 in \cite{Iw-arxiv} as
\[
\delta=\frac{1}{\tau}\left(\frac{1}{2s}\cdot\left\Vert \widehat{f}-\widehat{f}_{s}^{\rm opt}\right\Vert _{1}+3\left\Vert \f\right\Vert _{\infty}N^{-r}\right) = 3\left(\frac{1}{2s}\cdot\left\Vert \hf - \hf_{s}^{\rm opt}\right\Vert _{1}+3\left\Vert \f\right\Vert _{\infty}N^{-r}\right),
\]
and observe that any $\omega\in B=\left[-\left\lceil \frac{N}{2}\right\rceil ,\left\lfloor \frac{N}{2}\right\rfloor \right)\cap\mathbb{Z}$ that is reconstructed by Algorithm~\ref{Alg1} will have a Fourier series coefficient estimate $v_{\omega}$ that satisfies
\[
\left|v_{\omega}-\hf_{\omega}\right| = \left|v_{\omega}-\widehat{f}_{\omega}\right|\leq\sqrt{2}\cdot\delta.
\]
We can thus bound the approximation error by
\begin{align}
\begin{split}
\left\Vert \hf-\mathbf{v}\vert_{S}\right\Vert _{2} & \leq\left\Vert \hf-\hf\vert_{S}\right\Vert _{2}+\left\Vert \hf\vert_{S}-\mathbf{v}\vert_{S}\right\Vert _{2}\leq\left\Vert \hf-\hf\vert_{S}\right\Vert _{2}+2\sqrt{s}\cdot\delta\\
 & =\sqrt{\left\Vert \hf-\hf_{s}^{\rm opt}\right\Vert _{2}^{2}+\sum_{\omega\in R_{s}^{\rm opt}\left(\widehat{f}\right)\backslash S}\left|\widehat{f}_{\omega}\right|^{2}-\sum_{\tilde{\omega}\in S\backslash R_{s}^{\rm opt}\left(\widehat{f}\right)}\left|\widehat{f}_{\tilde{\omega}}\right|^{2}}+2\sqrt{s}\cdot\delta.
\end{split} 
 \label{equ:approximation_error_bound}
\end{align}
In order to make additional progress on \eqref{equ:approximation_error_bound} we must consider the possible magnitudes of $\mathbf{\widehat{f}}$ entries at indices in $S\backslash R_{s}^{\rm opt}\left(\widehat{f}\right)$ and $R_{s}^{\rm opt}\left(\widehat{f}\right)\backslash S$. Careful analysis (in line with the techniques employed in the proof of Theorem 7 of \cite{Iw-arxiv}) indicates that
\[
\sum_{\omega\in R_{s}^{\rm opt}\left(\widehat{f}\right)\backslash S}\left|\widehat{f}_{\omega}\right|^{2}-\sum_{\tilde{\omega}\in S\backslash R_{s}^{\rm opt}\left(\widehat{f}\right)}\left|\widehat{f}_{\tilde{\omega}}\right|^{2}\leq s\cdot\left(8\sqrt{2}+8\right)^{2}\cdot\delta^{2}.
\]
Therefore, in the worst possible case equation \eqref{equ:approximation_error_bound} will remain bounded by
\[
\left\Vert \hf-\mathbf{v}\vert_{S}\right\Vert _{2}\leq\sqrt{\left\Vert \hf-\hf_{s}^{\rm opt}\right\Vert _{2}^{2}+s\cdot\left(8\sqrt{2}+8\right)^{2}\cdot\delta^{2}}+2\sqrt{s}\cdot\delta\leq\left\Vert \hf - \hf_{s}^{\rm opt}\right\Vert _{2}+22\sqrt{s}\cdot\delta.
\]
The error bound stated in \eqref{equ:upper_bound} follows. 

The runtimes follow by observing that $c_2 = \mathcal{O} \left( \alpha\cdot \log^{\frac12} (N)\right) = \mathcal{O}\left( r^{\frac12}\cdot \log^{\frac12} (N) \right)$ as chosen in line 2 of Algorithm ~\ref{Alg1}, and for every choise of $q$ in line 4 of Algorithm ~\ref{Alg1}, all of the evaluations $\left\{ (\tilde{g}_q\ast f)(x_k) \right\}^m_{k=1}$ can be approximated very accurately in just $\mathcal{O}(m r \log N)$-time, where the number of samples $m$ is on the orders described in Theorem \ref{thm:StableRecov}.
\end{proof}

We are now ready to empirically evaluate Algorithm ~\ref{Alg1} with several different SFT algorithms $\mathcal{A}$ used in its line 9.

%%%%%%%%%%%%%%%%%%%%%%%%%%%%%%%%%%%%%%%%%%%%%%%%%%%%%%%
\section{Numerical Evaluation}
\label{sec:NumEval}

In this section we evaluate the performance of three new discrete SFT Algorithms resulting from Algorithm~\ref{Alg1}: DMSFT-4, DMSFT-6,\footnote{The code for both DMSFT variants is available at \url{https://sourceforge.net/projects/aafftannarborfa/}.}
 and CLW-DSFT.\footnote{The CLW-DSFT code is available at \url{www.math.msu.edu/~markiwen/Code.html}.}  All of them were developed by utilizing different SFT algorithms in line 9 of Algorithm~\ref{Alg1}.  Here DMSFT stands for the \textbf{D}iscrete \textbf{M}ichigan \textbf{S}tate \textbf{F}ourier \textbf{T}ransform algorithm.  Both DMSFT-4 and DMSFT-6 are implementations of Algorithm~\ref{Alg1} that use a randomized version of the SFT algorithm GFFT \cite{segal2013improved} in their line 9.\footnote{Code for GFFT is also available at \url{www.math.msu.edu/~markiwen/Code.html}.}  The only difference between DMSFT-4 and DMSFT-6 is how accurately each one estimates the convolution in line 7 of Algorithm~\ref{Alg1}: for DMSFT-4 we use $\kappa = 4$ in the partial discrete convolution in Lemma~\ref{lem:fewer_points_error} when approximating $\tilde{g}_q\ast f$ at each $x_k$, while for DMSFT-6 we always use $\kappa = 6$.  The CLW-DSFT stands for the \textbf{C}hristlieb \textbf{L}awlor \textbf{W}ang \textbf{D}iscrete \textbf{S}parse \textbf{F}ourier \textbf{T}ransform algorithm. It is an implementation of Algorithm~\ref{Alg1} that uses the SFT developed in \cite{clw} in its line 9, and $\kappa$ varying between $12$ and $20$ for its line 7 convolution estimates (depending on each input vector's Fourier sparsity, etc.).  All of DMSFT-4, DMSFT-6 and CLW-DSFT were implemented in C++ in order to empirically evaluate their runtime and noise robustness characteristics. 
 
We also compare these new implementations' runtime and robustness characteristics with FFTW 3.3.4\footnote{This code is available at \url{http://www.fftw.org/}} and sFFT 2.0\footnote{This code is available at \url{https://groups.csail.mit.edu/netmit/sFFT/}}. 
FFTW is the highly optimized FFT implementation which runs in $\mathcal{O}(N\log N)$-time for input vectors of length $N$.  All the standard discrete Fourier Transforms in the numerical experiments are performed using FFTW 3.3.4 with {\rm FFTW\_MEASURE} plan.  The sFFT 2.0 is a randomized discrete sparse Fourier Transform algorithm written in C++ which is both stable and robust to noise.  It was developed by Indyk et al. in \cite{sFFT2}.  Note that DMSFT-4, DMSFT-6, CLW-DSFT, and sFFT 2.0 are all randomized algorithms designed to approximate discrete DFTs that are approximately $s$-sparse. This means that all of them take both sparsity $s$ and size $N$ of the DFT's $\hf \in \mathbbm{C}^N$ they aim to recover as parameters. In contrast, FFTW can not utilize existing sparsity to its advantage.  Finally, all experiments are run on a Linux CentOS machine with 2.50GHz CPU and 16 GB of RAM. 

\subsection{Experiment Setup}
For the execution time experiments each trial input vector $\f \in \mathbbm{C}^N$ was generated as follows:  First $s$ frequencies were independently selected uniformly at random from $[0, N)\cap \mathbb{Z}$, and then each of these frequencies was assigned a uniform random phase with magnitude $1$ as its Fourier coefficient. The remaining frequencies' Fourier coefficients were then set to zero to form $\hf \in \mathbbm{C}^N$. Finally, the trial input vector $\f$ was then formed via an inverse DFT. 

For each pair of $s$ and $N$ the parameters in each randomized algorithm were chosen so that the probability of correctly recovering all $s$ energetic frequencies was at least 0.9 per trial input.  Every data point in a figure below corresponds to an average over 100 runs on 100 different trial input vectors of this kind.  It is worth mentioning that the parameter tuning process for DMSFT-4 and DMSFT-6 requires significantly less effort than for both CLW-DSFT and sFFT 2.0 since the DMSFT variants only have two parameters (whose default values are generally near-optimal).  

\subsection{Runtime as Input Vector Size Varies}
In Figure \ref{fig:fix_sparsity} we fixed the sparsity to $s=50$ and ran numerical experiments on 8 different input vector lengths $N$: $2^{16}$, $2^{18}$, ..., $2^{30}$. We then plotted the running time (averaged over 100 runs) for DMSFT-4, DMSFT-6, CLW-DSFT, sFFT 2.0, and FFTW.

\begin{figure}
  \centering
    \includegraphics[width=0.5\textwidth]{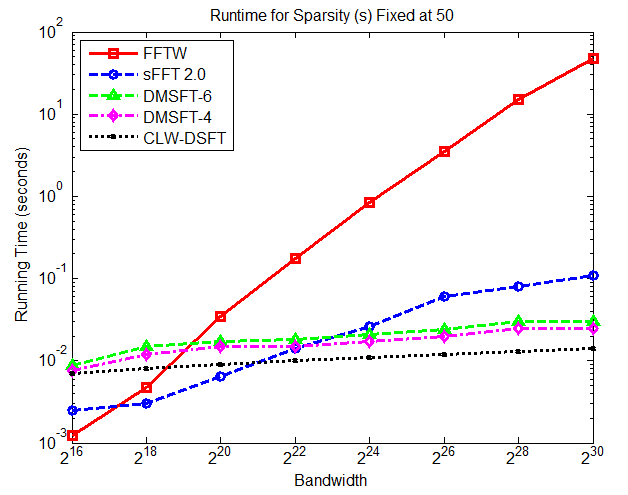}
  \caption{Runtime Comparison at Sparsity (s) Fixed at $50$}
  \label{fig:fix_sparsity}
\end{figure}

As expected, the runtime slope of all the SFT algorithms (i.e. DMSFT-4, DMSFT-6, CLW-DSFT, and sFFT 2.0) is less than the slope of FFTW as $N$ increases.  Although FFTW is fastest for vectors of small size, it becomes the slowest algorithm when the vector size $N$ is greater than $2^{20}$.  Among the randomized algorithms, sFFT 2.0 is the fastest one when $N$ is less than $2^{22}$, but DMSFT-4, DMSFT-6, and CLW-DSFT all outperform sFFT 2.0 with respect to runtime when the input vector's sizes are large enough. The CLW-DSFT implementation becomes faster than sFFT 2.0 when $N$ is approximately $2^{21}$ while DMSFT-4 and DMSFT-6 have better runtime performance than sFFT 2.0 when $N$ is greater than $2^{23}$. 

\begin{figure}
  \centering
    \includegraphics[width=0.51\textwidth]{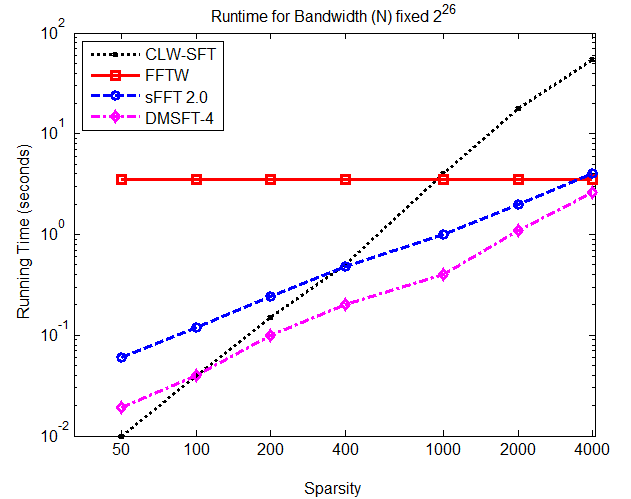}
  \caption{Runtime Comparison at Bandwidth (N) Fixed at $2^{26}$}
  \label{fix_bandwidwh}
\end{figure}

\subsection{Runtime as Sparsity Varies}
In Figure \ref{fix_bandwidwh} we fix the input vector lengths to $N = 2^{26}$ and run the numerical experiments on 7 different values of sparsity $s$: 50, 100, 200, 400, 1000, 2000, and 4000. As expected, the FFTW's runtime is constant as we increase the sparsity. The runtimes of DMSFT-4, CLW-DSFT, and sFFT 2.0 are all essentially linear in $s$.  Here DMSFT-6 has been excluded for ease of viewing/reference -- its runtimes lie directly above those of DMSFT-4 when included in the plot.  Looking at Figure 2 we can see the CLW-DSFT's runtime increases more rapidly with $s$ than that of DMSFT-4 and sFFT 2.0. The runtime of CLW-DSFT becomes the slowest one when sparsity is around 1000.  DMSFT-4 and sFFT 2.0 have approximately the same runtime slope as $s$ increases, and they both have good performance when the sparsity is large. However, DMSFT-4 maintains consistently better runtime performance than sFFT 2.0 for all sparsity values, and is the only algorithm in the plot that still faster than FFTW when the sparsity is 4000.  Indeed, when the sparsity is 4000 the average runtime of DMSFT-4 is $2.68$s and the average runtime of DMSFT-6 is $2.9$s. Both of them remain faster than FFTW ($3.47$s) and sFFT 2.0 ($3.96$s) for this large sparsity (though only DMSFT-4 has been included in the plot above).

\subsection{Robustness to Noise}
In our final set of experiments we test the noise robustness of DMSFT-4, DMSFT-6, CLW-DSFT, sFFT 2.0, and FFTW for different levels of Gaussian noise.  Here the size of each input vector is $N=2^{22}$ and sparsity is fixed at $s = 50$.  The test signals are then generated as before, except that Gaussian noise is added to $\f$ after it is constructed.  More specifically, we first generate $\f$ and then set $\f = \f + \n$ where each entry of $\n$, $n_j$, is an i.i.d. mean $0$ random complex Gaussian value.  The noise vector $\n$ is then rescaled to achieve each desired signal-to-noise ratio (SNR) considered in the experiments.\footnote{The SNR is defined as $SNR = 20\log \frac{\parallel \f \parallel_2}{\parallel \n \parallel_2}$, where $\f$ is the length $N$ input vector and $\n$ is the length $N$ noise vector.}

\begin{figure}
  \centering
    \includegraphics[width=0.51\textwidth]{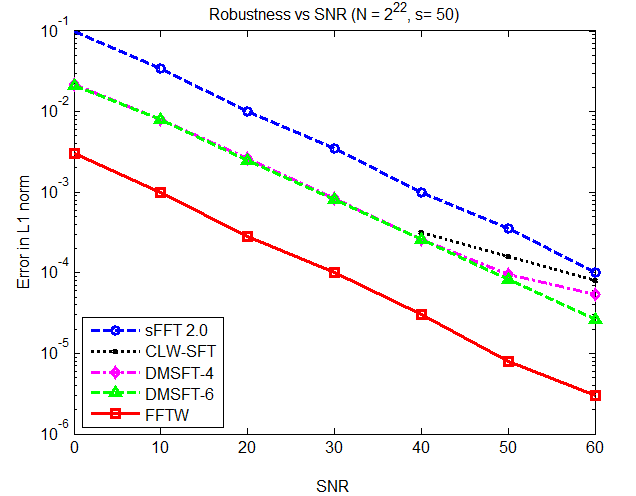}
  \caption{Robustness to Noise (Bandwidth (N) = $2^{22}$, Sparsity (s) = 50).}
  \label{noise}
\end{figure}

Recall that the the randomized algorithms compared herein (DMSFT-4, DMSFT-6, CLW-DSFT, and sFFT 2.0) are all tuned to guarantee exact recovery of $s$-sparse functions with probability at least 0.9 in all experiments. For our noise robustness experiments this ensures that
the correct frequency support, $S$, is found for at least 90 of the 100 trial signals used to generate each point plotted in Figure~\ref{noise}. We use average $L_1$ error to measure the noise robustness of each algorithm for each of these at least 90 trial runs.  The average $L_1$ error is defined as
\[
\textit{Average $L_1$ Error} = \frac{1}{s}\sum_{\omega \in S} \big|\hat{f}_{\omega} - z_{\omega} \big |
\]
where $S$ is the true frequency support of the input vector $\f$, $\hat{f}_{\omega}$ are the true input Fourier coefficients for all frequencies $\omega\in S$, and $z_{\omega}$ are their recovered approximations from each algorithm. Figure \ref{noise} graphs the averaged average $L_1$ error over the at least 90 trial signals where each method correctly identified $S$.

It can be seen in Figure \ref{noise} that DMSFT-4, DMSFT-6, sFFT 2.0, and FFTW are all robust to noise. As expected, FFTW has the best performance in this test. DMSFT-4 and DMSFT-6 are both more robust to noise when compared to sFFT 2.0. As for CLW-DSFT, it cannot guarantee a 0.9 probability of correctly recovering $S$ when the SNR is less than $40$ and so is not plotted for those SNR values.  This is due to the base energetic frequency identification methods of \cite{lawlor2013adaptive,clw} being inherently ill conditioned, though the CLW-DSFT results look better when compared to the true $\hf$ with respect to, e.g., earth mover's distance.  Frequencies are often estimated incorrectly by CLW-DSFT at higher noise levels, but when they are they are usually at least close enough to the true frequencies to be informative.  %%Unfortunately the same does not generally appear to be true of other inherently unstable implementations such as sFFT 3.0.

%%%%%%%%%%%%%%%%%%%%%%%%%%%%%%%%%%%%%%%%%%%%%%%%%%%%%%%
\section{Conclusion}
\label{sec:Conc}

Let $\mathcal{A}$ be a sublinear-time sparse FFT algorithm which utilizes unequally spaced samples from a given periodic function $f: [-\pi, \pi] \rightarrow \mathbbm{C}$ in order to rapidly approximate its sequence of Fourier series coefficients $\hat{f} \in \ell^2$.  In this paper we propose a generic method of transforming any such algorithm $\mathcal{A}$ into a sublinear-time sparse DFT algorithm which rapidly approximates $\hf$ from a given input vector $\f \in \mathbbm{C}^N$.  As a result we are able to construct several new sublinear-time sparse DFT algorithms from existing sparse Fourier algorithms which utilize unequally spaced function samples \cite{segal2013improved,Iw-arxiv,lawlor2013adaptive,clw}.  The best of these new algorithms is shown to outperform existing discrete sparse Fourier transform methods with respect to both runtime and noise robustness for large vector lengths $N$.  In addition, we also present several new theoretical discrete sparse FFT robust recovery guarantees.  These include the first known theoretical guarantees for entirely deterministic and discrete sparse DFT algorithms which hold for arbitrary input vectors $\f \in \mathbbm{C}^N$.

%%%%%%%%%%%%%%%%%%%%%%%%%%%%%%%%%%%%%%%%%%%%%%%%%%%%%%%
\bibliographystyle{abbrv}
\bibliography{DiscreteSFT}

\appendix

\section{Proof of Lemmas~\ref{lem:Decay_of_Periodic_Gaussian},~\ref{lem:Periodic_Gaussian_FC} and~\ref{lem:Periodic_Gaussian_Decay}}
\label{sec:proof_lems_1_2}
We will restate each lemma before its proof for ease of reference.

\begin{lem}[Restatement of Lemma~\ref{lem:Decay_of_Periodic_Gaussian}]
The $2\pi-$periodic Gaussian $g:\left[-\pi,\pi\right]\to\mathbb{R}^{+}$ has
\[
g\left(x\right) \leq \left( \frac{3}{c_1}+\frac{1}{\sqrt{2\pi}} \right)\e^{-\frac{x^{2}}{2c_{1}^{2}}}
\]
for all $x \in \left[-\pi,\pi\right]$.
\end{lem}

\begin{proof}
Observe that
\begin{align*}
c_{1}g\left(x\right) =\sum_{n=-\infty}^{\infty}\e^{-\frac{\left(x-2n\pi\right)^{2}}{2c_{1}^{2}}}
 & =\e^{-\frac{x^{2}}{2c_{1}^{2}}}+\e^{-\frac{\left(x-2\pi\right)^{2}}{2c_{1}^{2}}}+\e^{-\frac{\left(x+2\pi\right)^{2}}{2c_{1}^{2}}}+\sum_{\left|n\right|\geq2}\e^{-\frac{\left(x-2n\pi\right)^{2}}{2c_{1}^{2}}}\\
 %& \leq3\e^{-\frac{x^{2}}{2c_{1}^{2}}}+\sum_{n\geq2}\e^{-\frac{\left(x-2n\pi\right)^{2}}{2c_{1}^{2}}}+\sum_{n\geq2}\e^{-\frac{\left(x+2n\pi\right)^{2}}{2c_{1}^{2}}}\\
 & \leq3\e^{-\frac{x^{2}}{2c_{1}^{2}}}+\int_{1}^{\infty}\e^{-\frac{\left(x-2n\pi\right)^{2}}{2c_{1}^{2}}}dn+\int_{1}^{\infty}\e^{-\frac{\left(x+2n\pi\right)^{2}}{2c_{1}^{2}}}dn
\end{align*}
holds since the series above have monotonically decreasing positive terms, and $x\in\left[-\pi,\pi\right]$. 

Now, if $x\in\left[0,\pi\right]$ and $n\geq1$, one has
\[
\e^{-\frac{\left(2n+1\right)^{2}\pi^{2}}{2c_{1}^{2}}}\leq\e^{-\frac{\left(x+2n\pi\right)^{2}}{2c_{1}^{2}}}\leq\e^{-\frac{4n^{2}\pi^{2}}{2c_{1}^{2}}}\leq\e^{-\frac{\left(x-2n\pi\right)^{2}}{2c_{1}^{2}}}\leq\e^{-\frac{\left(2n-1\right)^{2}\pi^{2}}{2c_{1}^{2}}},
\]
which yields
\begin{align*}
c_{1}g\left(x\right) & %\leq3\e^{-\frac{x^{2}}{2c_{1}^{2}}}+\int_{1}^{\infty}\e^{-\frac{\pi^{2}\left(2n-1\right)^{2}}{2c_{1}^{2}}}dn+\int_{1}^{\infty}\e^{-\frac{4n^{2}\pi^{2}}{2c_{1}^{2}}}dn\\
 ~\leq~3\e^{-\frac{x^{2}}{2c_{1}^{2}}}+2\int_{1}^{\infty}\e^{-\frac{\pi^{2}\left(2n-1\right)^{2}}{2c_{1}^{2}}}dn
% =3\e^{-\frac{x^{2}}{2c_{1}^{2}}}+\int_{1}^{\infty}\e^{-\frac{\pi^{2}m^{2}}{2c_{1}^{2}}}dm\\
 ~=~ 3\e^{-\frac{x^{2}}{2c_{1}^{2}}}+\frac{1}{2}\left(\int_{-\infty}^{\infty}\e^{-\frac{\pi^{2}m^{2}}{2c_{1}^{2}}}dm-\int_{-1}^{1}\e^{-\frac{\pi^{2}m^{2}}{2c_{1}^{2}}}dm\right)\\
 & =3\e^{-\frac{x^{2}}{2c_{1}^{2}}}+\frac{c_{1}}{\sqrt{2\pi}}-\frac{1}{2}\int_{-1}^{1}\e^{-\frac{\pi^{2}m^{2}}{2c_{1}^{2}}}dm.
\end{align*}
Using Lemma \ref{lem:gaussian_integral_bound} to bound the last integral we can now get that
\begin{align*}
c_{1}g\left(x\right) & \leq 3\e^{-\frac{x^{2}}{2c_{1}^{2}}}+\frac{c_{1}}{\sqrt{2\pi}}-\frac{1}{2}\frac{\sqrt{2}c_{1}}{\pi}\sqrt{\pi\left(1-\e^{-\frac{\pi^{2}}{2c_{1}^{2}}}\right)}
  = 3\e^{-\frac{x^{2}}{2c_{1}^{2}}}+\frac{c_{1}}{\sqrt{2\pi}}\left(1-\sqrt{\left(1-\e^{-\frac{\pi^{2}}{2c_{1}^{2}}}\right)}\right)\\
& \leq 3\e^{-\frac{x^{2}}{2c_{1}^{2}}}+\frac{c_{1}}{\sqrt{2\pi}}\e^{-\frac{\pi^{2}}{2c_{1}^{2}}}
\leq3\e^{-\frac{x^{2}}{2c_{1}^{2}}}+\frac{c_{1}}{\sqrt{2\pi}}\e^{-\frac{x^{2}}{2c_{1}^{2}}}.
\end{align*}
Recalling now that $g$ is even we can see that this inequality will also hold for all $x \in [-\pi,0]$ as well.
\end{proof}

\begin{lem}[Restatement of Lemma~\ref{lem:Periodic_Gaussian_FC}]
The $2\pi-$periodic Gaussian $g:\left[-\pi,\pi\right]\to\mathbb{R}^{+}$ has
\[
\widehat{g}_{\omega} = \frac{1}{\sqrt{2\pi}}\e^{-\frac{c_1^2 \omega^2}{2}}
\]
for all $\omega\in\mathbb{Z}$.  Thus, $\widehat{g}=\left\{ \widehat{g}_{\omega}\right\} _{\omega\in\mathbb{Z}}\in\ell^{2}$ decreases monotonically as $|\omega|$ increases, and also has $\| \widehat{g} \|_{\infty} = \frac{1}{\sqrt{2 \pi}}$.
\end{lem}

\begin{proof}
Starting with the definition of the Fourier transform, we calculate 
\begin{eqnarray*}
\widehat{g}_{\omega} & = & \frac{1}{c_1}\sum_{n=-\infty}^{\infty}\frac{1}{2\pi}\int_{-\pi}^{\pi}\e^{-\frac{(x-2n\pi)^{2}}{2c_1^{2}}}\e^{-\i\omega x}~dx\\
 & = & \frac{1}{c_1}\sum_{n=-\infty}^{\infty}\frac{1}{2\pi}\int_{-\pi}^{\pi}\e^{-\frac{(x-2n\pi)^{2}}{2c_1^{2}}}\e^{-\i\omega(x-2n\pi)}~dx\\
 & = & \frac{1}{c_1}\sum_{n=-\infty}^{\infty}\frac{1}{2\pi}\int_{-\pi-2n\pi}^{\pi-2n\pi}\e^{-\frac{u{}^{2}}{2c_1^{2}}}\e^{-\i\omega u}~du\\
 & = & \frac{1}{2\pi c_1}\int_{-\infty}^{\infty}\e^{-\frac{u{}^{2}}{2c_1^{2}}} \e^{-\i\omega u}~du\\
 & = & \frac{c_1\sqrt{2\pi}}{2\pi c_1} \e^{-\frac{c_1^{2}\omega^{2}}{2}}\\
 & = & \frac{\e^{-\frac{c_1^{2}\omega^{2}}{2}}}{\sqrt{2\pi}}.
\end{eqnarray*}
The last two assertions now follow easily.
\end{proof}

\begin{lem}[Restatement of Lemma~\ref{lem:Periodic_Gaussian_Decay}]
Choose any $\tau \in \left(0, \frac{1}{\sqrt{2\pi}} \right)$, $\alpha \in \left[1, \frac{N}{\sqrt{\ln N}} \right]$, and $\beta \in \left(0 , \alpha \sqrt{\frac{\ln \left(1/\tau \sqrt{2\pi} \right)}{2}} ~\right]$. Let $c_1 = \frac{\beta \sqrt{\ln N}}{N}$ in the definition of the periodic Gaussian $g$ from \eqref{equ:Def_Periodic_Gaussian}. Then $\widehat{g}_{\omega} \in \left[\tau, \frac{1}{\sqrt{2\pi}} \right]$ for all $\omega \in \mathbb{Z}$ with $|\omega| \leq \Bigl\lceil \frac{N}{\alpha \sqrt{\ln N}}\Bigr\rceil$.
\end{lem}

\begin{proof}
By Lemma~\ref{lem:Periodic_Gaussian_FC} above it suffices to show that
\[
\frac{1}{\sqrt{2\pi}}\e^{-\frac{c_1^2 \left(\bigl\lceil \frac{N}{\alpha \sqrt{\ln N}}\bigr\rceil\right)^2}{2}} \geq \tau,
\]
which holds if and only if
\begin{eqnarray*}
%\e^{-\frac{c_1^2 \left(\bigl\lceil \frac{N}{\sqrt{\ln N}}\bigr\rceil\right)^2}{2}} &\geq & 10^{-l} \sqrt{2\pi},\\
c_1^2 \left(\left\lceil \frac{N}{\alpha \sqrt{\ln N}}\right\rceil\right)^2 &\leq & 2 \ln \left( \frac{1}{\tau \sqrt{2\pi}}\right)\\
%c_1^2 &\leq & \frac{2 \ln \left( \frac{10^l}{\sqrt{2\pi}}\right)}{\left(\bigl\lceil \frac{N}{k\sqrt{\log N}}\bigr\rceil\right)^2},\\
c_1 &\leq & \frac{\sqrt{2 \ln \left( \frac{1}{\tau \sqrt{2\pi}}\right)}}{\left\lceil \frac{N}{\alpha \sqrt{\ln N}}\right\rceil}.
\end{eqnarray*}
Thus, it is enough to have
\[
c_1 \leq  \frac{\sqrt{2 \ln \left( \frac{1}{\tau \sqrt{2\pi}}\right)}}{ \frac{N}{\alpha \sqrt{\ln N}} + 1} ~=~\frac{\alpha \sqrt{2 \ln \left( \frac{1}{\tau\sqrt{2\pi}}\right)\ln N}}{ N + \alpha \sqrt{\ln N}},
\]
or,
\[
c_1 = \frac{\beta \sqrt{\ln N}}{N} \leq \frac{\alpha \sqrt{2 \ln \left( \frac{1}{\tau\sqrt{2\pi}}\right)\ln N}}{2N} \leq \frac{\alpha \sqrt{2 \ln \left( \frac{1}{\tau\sqrt{2\pi}}\right)\ln N}}{ N + \alpha \sqrt{\ln N}}.
\]
This, in turn, is guaranteed by our choice of $\beta$.
\end{proof}

\section{Proof of Lemma~\ref{lem:SFT_Ident_Coef_Est_in_Noise} and Theorem~\ref{thm:StableRecov}}
\label{sec:proof_lem_3_thm_2}
We will restate Lemma~\ref{lem:SFT_Ident_Coef_Est_in_Noise} before its proof for ease of reference.

\begin{lem}[Restatement of Lemma~\ref{lem:SFT_Ident_Coef_Est_in_Noise}]
Let $s, \epsilon^{-1} \in \mathbb{N} \setminus \{ 1 \}$ with $(s/\epsilon) \geq 2$, and $\n \in \mathbb{C}^m$ be an arbitrary noise vector.  There exists a set of $m$ points $\left\{ x_k \right\}^m_{k=1} \subset [-\pi, \pi]$ such that Algorithm 3 on page 72 of \cite{Iw-arxiv}, when given access to the corrupted samples $\left\{ f(x_k) + n_k \right \}^m_{k=1}$, will identify a subset $S \subseteq B$ which is guaranteed to contain all $\omega \in B$ with
\begin{equation}
\left|\widehat{f}_{\omega} \right| > 4 \left( \frac{\epsilon \cdot \left\| \hf - \hf^{\rm opt}_{(s/\epsilon)} \right\|_1}{s} + \left\| \widehat{f} - \widehat{f}\vert_{B} \right\|_1 + \| \n \|_\infty \right).
\label{equ:GSFT_Identify2}
\end{equation}
Furthermore, every $\omega \in S$ returned by Algorithm 3 will also have an associate Fourier series coefficient estimate $z_{\omega} \in \mathbb{C}$ which is guaranteed to have 
\begin{equation}
\left| \widehat{f}_{\omega} - z_{\omega} \right| \leq \sqrt{2} \left( \frac{\epsilon \cdot \left\| \hf - \hf^{\rm opt}_{(s/\epsilon)} \right\|_1}{s} + \left\| \widehat{f} - \widehat{f}\vert_{B} \right\|_1 + \| \n \|_\infty \right).
\label{equ:GSFT_Estimate2}
\end{equation}
Both the number of required samples, $m$, and Algorithm 3's operation count are
\begin{equation}
\mathcal{O} \left( \frac{s^2 \cdot \log^4 (N)}{\log \left(\frac{s}{\epsilon} \right) \cdot \epsilon^2} \right).
\end{equation}

If succeeding with probability $(1-\delta) \in [2/3,1)$ is sufficient, and $(s/\epsilon) \geq 2$, the Monte Carlo variant of Algorithm 3 referred to by Corollary 4 on page 74 of \cite{Iw-arxiv} may be used.  This Monte Carlo variant reads only a randomly chosen subset of the noisy samples utilized by the deterministic algorithm,
$$\left\{ f(\tilde{x}_k) + \tilde{n}_k \right \}^{\tilde{m}}_{k=1} \subseteq \left\{ f(x_k) + n_k \right \}^m_{k=1},$$
yet it still outputs a subset $S \subseteq B$ which is guaranteed to simultaneously satisfy both of the following properties with probability at least $1-\delta$: 
\begin{enumerate}
\item [(i)] $S$ will contain all $\omega \in B$ satisfying \eqref{equ:GSFT_Identify2}, and 
\item [(ii)] all $\omega \in S$ will have an associated coefficient estimate $z_{\omega} \in \mathbb{C}$ satisfying \eqref{equ:GSFT_Estimate2}.
\end{enumerate}  
Finally, both this Monte Carlo variant's number of required samples, $\tilde{m}$, as well as its operation count will also always be
\begin{equation}
\mathcal{O} \left( \frac{s}{\epsilon} \cdot \log^3 (N) \cdot \log \left( \frac{N}{\delta} \right) \right).
\end{equation}
\end{lem}

\begin{proof}
The proof of this lemma involves a somewhat tedious and uninspired series of minor modifications to various results from \cite{Iw-arxiv}.  In what follows we will outline the portions of that paper which need to be changed in order to obtain the stated lemma.  Algorithm 3 on page 72 of \cite{Iw-arxiv} will provide the basis of our discussion.

In the first paragraph of our lemma we are provided with $m$-contaminated evaluations of $f$, $\left\{ f(x_k) + n_k \right \}^m_{k=1}$, at the set of $m$ points $\left\{ x_k \right\}^m_{k=1} \subset [-\pi, \pi]$ required by line 4 of Algorithm 1 on page 67 of \cite{Iw-arxiv}.  These contaminated evaluations of $f$ will then be used to approximate the vector ${\it \mathcal{G}_{\lambda,K} \tilde{\psi} {\bf A}} \in \mathbb{C}^m$ in line 4 of Algorithm 3.  More specifically, using $(18)$ on page 67 of \cite{Iw-arxiv} one can see that each $\left({\it \mathcal{G}_{\lambda,K} \tilde{\psi} {\bf A}} \right)_j \in \mathbb{C}$ is effectively computed via a DFT 
\begin{equation}
\left({\it \mathcal{G}_{\lambda,K} \tilde{\psi} {\bf A}} \right)_j = \frac{1}{s_j} \sum^{s_j - 1}_{k=0} f \left(-\pi + \frac{2 \pi k}{s_j} \right) \e^{\frac{-2 \pi \i k h_j}{s_j}}
\label{equ:DFT_GKlam_Comp}
\end{equation}
for some integers $0 \leq h_j < s_j$.  Note that we are guaranteed to have noisy evaluations of $f$ at each of these points by assumption.  That is, we have $f \left( x_{j,k} \right) + n_{j,k}$ for all $x_{j,k} := -\pi + \frac{2 \pi k}{s_j}$, $k = 0, \dots, s_j - 1$. 

We therefore approximate each $\left({\it \mathcal{G}_{\lambda,K} \tilde{\psi} {\bf A}} \right)_j$ via an approximate DFT
as per \eqref{equ:DFT_GKlam_Comp} by
$$E_j :=  \frac{1}{s_j} \sum^{s_j - 1}_{k=0} \left( f \left( x_{j,k} \right) + n_{j,k} \right)\e^{\frac{-2 \pi \i k h_j}{s_j}}.$$
One can now see that 
\begin{equation}
\left| E_j - \left({\it \mathcal{G}_{\lambda,K} \tilde{\psi} {\bf A}} \right)_j \right| = \left| \frac{1}{s_j} \sum^{s_j - 1}_{k=0} n_{j,k} \e^{\frac{-2 \pi \i k h_j}{s_j}} \right| \leq \frac{1}{s_j} \sum^{s_j - 1}_{k=0} \left| n_{j,k} \right| \leq \| \n \|_{\infty}
\label{equ:E_approx_error_to_GKlam}
\end{equation}
holds for all $j$.
Every entry of both ${\it \mathcal{E}_{s_1,K} \tilde{\psi} {\bf A}}$ and ${\it \mathcal{G}_{\lambda,K} \tilde{\psi} {\bf A}}$ referred to in Algorithm 3 will therefore be effectively replaced by its corresponding $E_j$ estimate.  Thus, the lemma we seek to prove is essentially obtained by simply incorporating the additional error estimate \eqref{equ:E_approx_error_to_GKlam} into the analysis of Algorithm 3 in \cite{Iw-arxiv} wherever an ${\it \mathcal{E}_{s_1,K} \tilde{\psi} {\bf A}}$ or ${\it \mathcal{G}_{\lambda,K} \tilde{\psi} {\bf A}}$ currently appears.

To show that lines 6 -- 14 of Algorithm 3 will identify all $\omega \in B$ satisfying \eqref{equ:GSFT_Identify2} we can adapt the proof of Lemma 6 on page 72 of \cite{Iw-arxiv}.  Choose any $\omega \in B$ you like.  Lemmas 3 and 5 from \cite{Iw-arxiv} together with \eqref{equ:E_approx_error_to_GKlam} above ensure that both
\begin{align}
\left| E_j - \widehat{f}_{\omega} \right| &\leq \left| E_j - \left({\it \mathcal{G}_{\lambda,K} \tilde{\psi} {\bf A}} \right)_j \right| + \left| \left({\it \mathcal{G}_{\lambda,K} \tilde{\psi} {\bf A}} \right)_j - \widehat{f}_{\omega} \right| \nonumber \\ &\leq  \| \n \|_{\infty} + \frac{\epsilon \cdot \left\| \hf - \hf^{\rm opt}_{(s/\epsilon)} \right\|_1}{s} + \left\| \widehat{f} - \widehat{f}\vert_{B} \right\|_1
\label{eq:Ej_apprx_Fourier_coef}
\end{align}
and
\begin{align}
\left| E_{j'} - \widehat{f}_{\omega} \right| &\leq \left| E_{j'} - \left({\it \mathcal{E}_{s_1,K} \tilde{\psi} {\bf A}} \right)_{j'} \right| + \left| \left({\it \mathcal{E}_{s_1,K} \tilde{\psi} {\bf A}} \right)_{j'} - \widehat{f}_{\omega} \right| \nonumber \\ &\leq  \| \n \|_{\infty} + \frac{\epsilon \cdot \left\| \hf - \hf^{\rm opt}_{(s/\epsilon)} \right\|_1}{s} + \left\| \widehat{f} - \widehat{f}\vert_{B} \right\|_1
\end{align}
hold for more than half of the $j$ and $j'$-indexes that Algorithm 3 uses to approximate $\widehat{f}_{\omega}$.  The rest of the proof of Lemma 6 now follows exactly as in \cite{Iw-arxiv} after the $\delta$ at the top of page 73 is redefined to be $\delta := \frac{\epsilon \cdot \left\| \hf - \hf^{\rm opt}_{(s/\epsilon)} \right\|_1}{s} + \left\| \widehat{f} - \widehat{f}\vert_{B} \right\|_1 + \| \n \|_\infty$, each $\left({\it \mathcal{G}_{\lambda,K} \tilde{\psi} {\bf A}} \right)_j$ entry is replaced by $E_j$, and each $\left( {\it \mathcal{E}_{s_1,K} \tilde{\psi} {\bf A}} \right)_{j'}$ entry is replaced by $E_{j'}$.

Similarly, to show that lines 15 -- 18 of Algorithm 3 will produce an estimate $z_{\omega} \in \mathbb{C}$ satisfying \eqref{equ:GSFT_Estimate2} for every $\omega \in S$ one can simply modify the first few lines of the proof of Theorem 7 in Appendix F of \cite{Iw-arxiv}.  In particular, one can redefine $\delta$ as above, replace the appearance of each $\left({\it \mathcal{G}_{\lambda,K} \tilde{\psi} {\bf A}} \right)_j$ entry by $E_j$, and then use \eqref{eq:Ej_apprx_Fourier_coef}.  The bounds on the runtime follow from the last paragraph of the proof of Theorem 7 in Appendix F of \cite{Iw-arxiv} with no required changes.  
To finish, we note that the second paragraph of the lemma above follows from a completely analogous modification of the proof of Corollary 4 in Appendix G of \cite{Iw-arxiv}.

\end{proof}

\subsection{Proof of Theorem~\ref{thm:StableRecov}}

To get the first paragraph of Theorem~\ref{thm:StableRecov} one can simply utilize the proof of Theorem 7 exactly as it is written in Appendix F of \cite{Iw-arxiv} after redefining $\delta$ as above, and then replacing the appearance of each $\left({\it \mathcal{G}_{\lambda,K} \tilde{\psi} {\bf A}} \right)_j$ entry with its approximation $E_j$.  Once this has been done, equation (42) in the proof of Theorem 7 can then be taken as a consequence of Lemma~\ref{lem:SFT_Ident_Coef_Est_in_Noise} above.  In addition, all references to Lemma 6 of \cite{Iw-arxiv} in the proof can then also be replaced with appeals to Lemma~\ref{lem:SFT_Ident_Coef_Est_in_Noise} above.  To finish, the proof of Corollary 4 in Appendix G of \cite{Iw-arxiv} can now be modified in a completely analogous fashion in order to prove the second paragraph of Theorem~\ref{thm:StableRecov}.

%%%%%%%%%%%%%%%%%%%%%%%%%%%%%%%%%%%%%%%%%%%%%%%%%%%%%%%
\section*{Acknowledgements}

M.A. Iwen, R. Zhang, and S. Merhi were all supported in part by NSF DMS-1416752.  The authors would like to thank Aditya Viswanathan for helpful comments and feedback on the first draft of the paper.

\end{document}